\documentclass[12pt,oneside,reqno]{amsart}
\usepackage{etex}
\usepackage{pifont}
\usepackage{soul}
\usepackage[dvips]{graphicx}
\usepackage{amsmath,latexsym,amsthm,cmap,indentfirst,setspace,ccaption,amssymb,color,amscd,mathrsfs,euscript,tikz-cd}

\usepackage{xcolor}

\usepackage{listings}
\usepackage{lstlang0}

\definecolor{codepurple}{rgb}{0.1,0.1,0.5}

\definecolor{codegray}{rgb}{0.2,0.5,0.5}
\definecolor{codegreen}{rgb}{0.12,0.55,0.6}
\definecolor{backcolour}{rgb}{0.93,0.97,0.95}

\lstdefinestyle{mystyle}{
	backgroundcolor=\color{backcolour},   
	commentstyle=\color{codegreen},
	keywordstyle=\color{magenta},
	numberstyle=\tiny\color{codegray},
	stringstyle=\color{codepurple},
	basicstyle=\ttfamily\footnotesize,
	breakatwhitespace=false,         
	breaklines=true,                 
	captionpos=t,                    
	keepspaces=true,                 
	numbers=left,                    
	numbersep=8pt,                  
	showspaces=false,                
	showstringspaces=false,
	showtabs=false,                  
	tabsize=2
}

\usepackage{makecell}

\usepackage[normalem]{ulem}

\usepackage{tikz}
\usetikzlibrary{decorations.markings,arrows}

\newcommand{\nicecolor}{Navy}
\usepackage[pdfauthor={A. Pinardin, A. Sarikyan \& E. Yasinsky}, backref=page, colorlinks, linktocpage, citecolor=\nicecolor, linkcolor=\nicecolor, urlcolor=\nicecolor]{hyperref}
\usepackage{array}

\usepackage{comment}

\usepackage{amssymb,lipsum}
\usepackage{float,xypic}
\usepackage{cmap,longtable,booktabs}
\usepackage[all,cmtip]{xy}

\usepackage{colortbl}

\usepackage{geometry}
\usepackage{caption, subcaption}

\makeatletter\@addtoreset{equation}{section}
\makeatother 

\usepackage{verbatim}
\hypersetup{
    bookmarks=true,         
    unicode=true,          
    pdftoolbar=true,        
    pdfmenubar=true,        
    pdffitwindow=false,     
    pdfnewwindow=true,      
    colorlinks=true,       
    linkcolor=blue,          
    citecolor=blue,        
    filecolor=magenta,      
    urlcolor=cyan           
}

\DeclareMathOperator{\Cr}{Cr}

\DeclareMathOperator{\tr}{tr}
\DeclareMathOperator{\Bir}{Bir}

\DeclareMathOperator{\rk}{rk}
\DeclareMathOperator{\Pic}{Pic}
\DeclareMathOperator{\diag}{diag}

\DeclareMathOperator{\Aut}{Aut}

\DeclareMathOperator{\Ker}{Ker}
\DeclareMathOperator{\PGL}{PGL}
\DeclareMathOperator{\GL}{GL}
\DeclareMathOperator{\SL}{SL}
\DeclareMathOperator{\PSL}{PSL}

\DeclareMathOperator{\Proj}{Proj}

\DeclareMathOperator{\PO}{PO}

\DeclareMathOperator{\Center}{Z}
\DeclareMathOperator{\Centralizer}{C}
\DeclareMathOperator{\Cohom}{H}

\DeclareMathOperator{\Hom}{Hom}
\DeclareMathOperator{\Hol}{Hol}

\DeclareMathOperator{\Burn}{Burn}
\DeclareMathOperator{\Symb}{Symb}
\DeclareMathOperator{\Inn}{Inn}

\DeclareMathOperator{\Norm}{N}

\theoremstyle{plain}
\newtheorem{theorem}[equation]{Theorem}
\newtheorem{lemma}[equation]{Lemma}
\newtheorem{lemma-def}[equation]{Lemma-Definition}
\newtheorem{corollary}[equation]{Corollary}
\newtheorem{proposition}[equation]{Proposition}
\newtheorem*{maintheorem*}{Main Theorem}
\newtheorem*{corollary*}{Corollary}
\newtheorem*{theorem*}{Theorem}

\theoremstyle{definition}
\newtheorem{example}[equation]{Example}
\newtheorem*{example*}{Example}
\newtheorem{definition}[equation]{Definition}
\newtheorem{definition-example}[equation]{Definition-Example}
\newtheorem{definition-proposition}[equation]{Definition-Proposition}
\newtheorem{question}[equation]{Question}
\newtheorem*{question*}{Question}
\newtheorem*{questionprime*}{Question'}

\newtheorem*{problem*}{Problem}

\newtheorem*{conjecture*}{Conjecture}

\theoremstyle{remark}
\newtheorem{remark}[equation]{Remark}
\newtheorem*{remark*}{Remark}
\newtheorem*{convention*}{Convention}
\newtheorem*{conventions*}{Conventions}
\newtheorem*{observation*}{Observation}

\newtheorem{notation}[equation]{Notation}

\definecolor{white}{HTML}{FFFFFF}
\definecolor{light-gray}{HTML}{E5E4E2}
\definecolor{light-green}{HTML}{E3F4F2}

\newcommand\iso{\stackrel{\sim}{\to}}

\DeclareFontFamily{U}{mathb}{\hyphenchar\font45}
\DeclareFontShape{U}{mathb}{m}{n}{
	<5> <6> <7> <8> <9> <10> gen * mathb
	<10.95> mathb10 <12> <14.4> <17.28> <20.74> <24.88> mathb12
}{}
\DeclareSymbolFont{mathb}{U}{mathb}{m}{n}
\DeclareMathSymbol{\righttoleftarrow}{3}{mathb}{"FD}

\newcommand{\actsfromleft}{\mathrel{\reflectbox{$\righttoleftarrow$}}}
\newcommand{\actsfromright}{\righttoleftarrow}

\makeatletter
\g@addto@macro{\endabstract}{\@setabstract}
\newcommand{\authorfootnotes}{\renewcommand\thefootnote{\@fnsymbol\c@footnote}}%
\makeatother

\title[Linearization problem in the plane Cremona group]{Linearization problem for finite subgroups of the plane Cremona group}

\author{Antoine Pinardin}
\address{
	School of Mathematics, The University of Edinburgh, Edinburgh EH9 3JZ, UK}
\email{antoine.pinardin@ed.ac.uk}

\author{Arman Sarikyan}
\address{
	London Institute for Mathematical Sciences, Royal Institution, 21 Albermarle St, London W1S 4BS, UK}
\email{ars@lims.ac.uk}

\author{Egor Yasinsky}
\address{
	L'Institut de Math\'{e}matiques de Bordeaux, Universit\'{e} de Bordeaux, 351 Cours de la Lib\'{e}ration,
	33405 Talence Cedex, France}
\email{egor.yasinsky@u-bordeaux.fr}


\newcommand{\I}{\ensuremath{\mathrm{I}}}
\newcommand{\II}{\ensuremath{\mathrm{II}}}
\newcommand{\III}{\ensuremath{\mathrm{III}}}
\newcommand{\IV}{\ensuremath{\mathrm{IV}}}

\newcommand{\CC}{\mathbb C}
\newcommand{\QQ}{\mathbb Q}
\newcommand{\FF}{\mathbb F}
\newcommand{\id}{\mathrm{id}}

\newcommand{\kk}{\mathbf k}
\newcommand{\RR}{\mathbb R}
\renewcommand{\AA}{\mathbb A}
\newcommand{\PP}{\mathbb P}
\newcommand{\ZZ}{\mathbb Z}
\newcommand{\KK}{\mathbf K}
\newcommand{\pt}{\mathrm{pt}}

\newcommand{\Torus}{\mathbb T}

\newcommand{\Sym}{\mathrm S}
\newcommand{\Alt}{\mathrm A}
\newcommand{\Dih}{\mathrm D}

\newcommand{\Cyc}{\mathrm C}
\newcommand{\Klein}{\mathrm V}
\newcommand{\Quat}{\mathrm Q}
\newcommand{\Frobenius}{\mathrm F}

\newcommand{\A}{\mathrm{A}}
\newcommand{\B}{\mathrm{B}}
\newcommand{\C}{\mathrm{C}}
\newcommand{\D}{\mathrm{D}}
\newcommand{\E}{\mathrm{E}}
\newcommand{\F}{\mathrm{F}}
\newcommand{\R}{\mathrm{R}}
\newcommand{\Rot}{\mathrm{R}}
\newcommand{\M}{\mathrm{M}}
\newcommand{\N}{\mathrm{N}}
\newcommand{\T}{\mathrm{T}}
\renewcommand*{\P}{\mathrm{P}}
\renewcommand*{\I}{\mathrm{I}}

\renewcommand*{\swap}{\sigma}

\def\dashmapsto{\mapstochar\dashrightarrow}

\makeatletter
\def\@tocline#1#2#3#4#5#6#7{\relax
	\ifnum #1>\c@tocdepth 
	\else
	\par \addpenalty\@secpenalty\addvspace{#2}%
	\begingroup \hyphenpenalty\@M
	\@ifempty{#4}{%
		\@tempdima\csname r@tocindent\number#1\endcsname\relax
	}{%
		\@tempdima#4\relax
	}%
	\parindent\z@ \leftskip#3\relax \advance\leftskip\@tempdima\relax
	\rightskip\@pnumwidth plus4em \parfillskip-\@pnumwidth
	#5\leavevmode\hskip-\@tempdima
	\ifcase #1
	\or\or \hskip 3em \or \hskip 4em \else \hskip 5em \fi%
	#6\nobreak\relax
	\hfill\hbox to\@pnumwidth{\@tocpagenum{#7}}\par
	\nobreak
	\endgroup
	\fi}
\makeatother

\begin{document}
	
\maketitle

\begin{abstract}
We give a complete solution of the linearization problem in the plane Cremona group over an algebraically closed field of characteristic zero. 
\end{abstract}

\tableofcontents

\section{The linearization problem: brief history of previous work}\label{sec:intro}

\subsection{Finite subgroups of the plane Cremona group} Given a birational self-map of the projective plane, for example
\begin{equation*}
\varphi\colon [x:y:z]\mapsto [x(z-y):z(x-y):xz],
\end{equation*}
it is not at all obvious whether this map is conjugate to a \emph{linear} one within the entire group of birational transformations of the plane, known as the \emph{Cremona group} of rank 2. The map~$\varphi$ is indeed non-regular, as it has the indeterminacy points $[1:0:0],[0:1:0],[0:0:1]$. As a matter of fact, $\varphi$ is conjugate over the field $\QQ(\sqrt{5})$ to a linear map \cite{BeauvilleBlanc}. Letting $\zeta=(1+\sqrt{5})/2$, the birational involution
\[
[x:y:z]\mapsto \left [(x-\zeta y)(y-z)(z-\zeta x):(\zeta^{-1}x-y)(\zeta^2y-z)(z-x):(x-y)(\zeta^2y-z)(z-\zeta x) \right ]
\]
conjugates $\varphi$ to 
\[
[x:y:z]\mapsto [y-\zeta^{-2}z:y-\zeta^{-1}x:y].
\]
Naturally, the problem becomes even more complicated when considering arbitrary finite, not necessarily cyclic, groups and/or specific base fields. Regarding the latter, in this paper we will always work over an algebraically closed field $\kk$ of characteristic zero, say $\kk=\CC$. In this case, the very description of all finite subgroups of the plane Cremona group, denoted $\Cr_2(\kk)$ hereafter, becomes a challenging task. This problem traces its origins to E. Bertini's work on involutions in $\Cr_2(\CC)$. Bertini identified three types of conjugacy classes, now referred to as de Jonqui\`{e}res, Geiser, and Bertini involutions. However, his classification was incomplete, and his proofs lacked rigor. Progress continued in 1895 with S. Kantor and A. Wiman, who provided a more detailed description of finite subgroups in $\Cr_2(\CC)$, though their work was not entirely accurate too. 

The modern approach to this problem began with a crucial insight of Yu.~Manin and V.~Iskovskikh, who discovered a connection between the conjugacy classes of finite subgroups in the Cremona groups and classification of $G$-varieties up to $G$-birational equivalence between them (see Section \ref{sec: G-surfaces}). In 2000, L. Bayle and A. Beauville further developed this approach in their study of involutions \cite{BayleBeauville}. Later, T. de Fernex extended the classification to subgroups of prime order \cite{deFernex}, while J. Blanc classified finite abelian subgroups in $\Cr_2(\CC)$, see \cite{BlancDissertation}. The most comprehensive description of arbitrary finite subgroups in $\Cr_2(\CC)$ was obtained by I. Dolgachev and V. Iskovskikh in their seminal work \cite{DolgachevIskovskikh}.

What Dolgachev and Iskovskikh did was not only the classification of finite subgroups of $\Cr_2(\CC)$ up to isomorphism, but also the description of many conjugacy classes, which is clearly a much more delicate problem. For instance, it was established by Beauville and Bayle \cite{BayleBeauville} that the conjugacy classes of involutions in $\Cr_2(\CC)$ are in one-to-one correspondence with the isomorphism classes of their fixed curves. In particular, an involution is \emph{linearizable}, i.e. conjugate in $\Cr_2(\CC)$ to a linear one, if and only if its fixed curve is rational (notice that in this case it is crucial that the base field is algebraically closed: as shown in a recent work \cite{CheltsovMangolteYasinskyZimmermann}, the aforementioned correspondence fails over a non-closed field, such as $\RR$).

\medskip

Nevertheless, despite numerous specific results (see, for example \cite{DolgachevIskovskikh,BlancLinearisation}), the following basic question remained open: \emph{which finite subgroups of $\Cr_2(\CC)$ are linearizable?} The goal of this paper is to provide a complete answer to this question. Before presenting our results, let us place the problem in a broader context and describe some existing approaches to solving it.

\subsection{The Bogomolov--Prokhorov invariant}

Consider the tower of the Cremona groups
\[
\Cr_1(\kk)\subset\Cr_2(\kk)\subset\Cr_3(\kk)\subset\Cr_4(\kk)\subset\ldots,
\]
where $\Cr_i(\kk)$ denotes the group of birational self-maps of $\PP_\kk^{i}$, and each embedding $\Cr_i(\kk)\subset\Cr_{i+1}(\kk)$ is induced by adding a new variable. Two subgroups $G_1\subset\Cr_n(\kk)$ and $G_2\subset\Cr_m(\kk)$ are called \emph{stably conjugate} if they are conjugate in some larger Cremona group $\Cr_N(\kk)\supset\Cr_n(\kk),\Cr_m(\kk)$, where $N\geqslant n,m$. 

As will be recalled in Section \ref{subsec: G-MMP}, any finite subgroup $G\subset\Cr_n(\kk)$ is induced by a \emph{biregular} action on some smooth projective rational variety $X$. The process of replacing the initial rational action of $G$ by a regular one on $X$ is called a \emph{regularization} of $G$, and then one speaks about \emph{$G$-variety}, i.e. the pair $(X,G)$, where $G\subset\Aut(X)$. We say that two $G$-varieties $(X,G)$ and $(Y,G)$ are \emph{stably birational} if there exists a $G$-birational map
\[
X\times\PP^n\dashrightarrow Y\times\PP^m,
\]
where actions on $\PP^n$ and $\PP^m$ are trivial (so, the case $G=\id$ corresponds to classical notion of stable birationality). If $(X,G)$ is stably birational to $(\PP^N,G)$, then we say that $G$ is \emph{stably linearizable}. The following question can be viewed as an analogue of the famous Zariski cancellation problem in our setting:

\begin{question}\label{question: Zariski}
	Let $G\subset\Cr_2(\kk)$ be a stably linearizable finite subgroup. Is $G$ linearizable?
\end{question}

As was observed by F. Bogomolov and Yu. Prokhorov \cite{BogomolovProkhorov}, for a smooth projective $G$-variety $X$, the cohomology group $\Cohom^1(G,\Pic(X))$ is a $G$-birational invariant (in the context of rationality
questions over non-closed fields this observation goes back to Yu. Manin). Moreover, one has the following.
	
\begin{theorem*}[{\cite[Corollary 2.5.1]{BogomolovProkhorov}}]
	If $(X,G)$ and $(Y,G)$ are projective, smooth, and stably birational, then
	\[
	\Cohom^1(G,\Pic(X))\simeq\Cohom^1(G,\Pic(Y)).
	\]
\end{theorem*}	

\begin{corollary*}[{\cite[Corollary 2.5.2]{BogomolovProkhorov}}]
	If $(X,G)$ is stably linearizable, then $\Cohom^1(H,\Pic(X))=0$ for any subgroup $H\subset G$.	
\end{corollary*}

If the latter condition holds, the $G$-variety $X$ is said to be \emph{$\Cohom^1$-trivial}. It turns out that in some cases, the invariant $\Cohom^1(G,\Pic(X))$ can be computed in
terms of $G$-fixed locus:

\begin{theorem}[{\cite[Theorem 1.1, Corollary 1.2]{BogomolovProkhorov}}] Let a finite cyclic group $G$ of prime order $p$ act on a non-singular projective rational surface $X$. Assume that $G$ fixes (point-wise) a curve of genus $g > 0$. Then
	\[
	\Cohom^1(G,\Pic(X))\simeq(\ZZ/p\ZZ)^{2g}.
	\]
	Moreover, the following are equivalent:
	\begin{enumerate}
		\item $G$ is $\Cohom^1$-trivial;
		\item $(X,G)$ is linearizable;
		\item $(X,G)$ is stably linearizable.
	\end{enumerate}
\end{theorem}
In particular, de Jonqui\`{e}res, Bertini and Geiser involutions are not stably linearizable: they all fix a curve of positive genera.

In \cite{ProkhorovStableConjugacy}, Prokhorov further applies the invariant introduced by Bogomolov and himself and indicates where to look for negative answers to Question \ref{question: Zariski}. He shows, for example, that a del Pezzo surface $S$ acted on by a finite group $G$ with $\Pic(S)^G\simeq\ZZ$, is $\Cohom^1$-trivial if and only if (i) either $K_S^2\geqslant 5$, (ii) or $S$ is the quartic 
\[
x_1^2+\omega_3 x_2^2+\omega_3^2 x_3^2+x_4^2=x_1^2+\omega_3^2x_2^2+\omega_3 x_3^2+x_5^2=0
\]
in $\PP^4$, acted on by the group $G\simeq\Cyc_3\rtimes\Cyc_4$ generated by two automorphisms
\[
(x_1,x_2,x_3,x_4,x_5)\mapsto (x_2,x_3,x_1,\omega_3 x_4,\omega_3^2x_5),\ \ (x_1,x_2,x_3,x_4,x_5)\mapsto (x_1,x_3,x_2,-x_5,x_4),
\]
where $\omega_3=\exp(2\pi i/3)$. Our Main Theorem below shows that only a few groups in these two cases (i) and (ii) are linearizable. Constructions of \emph{stable} linearizability in the remaining non-linearizable cases are sporadic and remain an open problem even in dimension 2. For instance, as of the time of writing, it is unknown whether the group $G\simeq\Cyc_3\rtimes\Cyc_4$ mentioned above is stably linearizable. Some of these constructions are mentioned below, see Remarks~\ref{rem: dP5 stable linearizability}, \ref{rem: Cayley groups}, and~\ref{rem: dP6 S4 stably linearizable action}.

\subsection{The Burnside formalism}\label{subsec: Burnside}

Another very recent technique for distinguishing birational actions of finite groups is the \emph{Burnside group formalism}, introduced in the work \cite{KreschTschinkelBurnsideVolume} of A.~Kresch and Yu.~Tschinkel and generalizing the \emph{birational symbols groups} of M. Kontsevich, V. Pestun and Yu. Tschinkel \cite{KontsevichPestumTschinkel}. Here, we will limit ourselves to a rough sketch of this approach.

In \cite{KreschTschinkelBurnsideVolume}, the \emph{symbols group} $\Symb_n(G)$ was defined as the free abelian group generated by \emph{symbols} $(H,R\actsfromleft \KK,\beta)$. Here, $H\subset G$ is an abelian subgroup, $R\subset\Centralizer_G(H)/H$ is a subgroup, where $\Centralizer_G(H)$ denotes the centralizer of~$H$, $\KK/\kk$ is a finitely generated extension of transcendence degree $d\leqslant n$ faithfully acted on by $R$, and $\beta=(b_1,\ldots,b_{n-d})$ is a sequence of nontrivial characters of $H$, generating $\Hom(H,\kk^*)$. The quotient of $\Symb_n(G)$ by some involved \emph{conjugation} and \emph{blow-up relations} \cite[Section 4]{KreschTschinkelBurnsideVolume} gives the \emph{equivariant Burnside group} $\Burn_n(G)$. 

Now, let $(X,G)$ be a smooth $n$-dimensional projective variety with a generically free action of~$G$. By \cite[Theorem 3.2]{ReichsteinYoussin} or \cite[Section 7.2]{HassettKreschTschinkel}, the action of $G$ can be always brought to a \emph{standard form} via equivariant blow-ups. This means that there is a $G$-invariant simple normal crossing divisor $\Delta\subset X$ such that $G$ acts freely on $X\setminus \Delta$, and for every $g\in G$ and every irreducible component $D$ of $\Delta$ one has either $g(D)=D$ or $g(D)\cap D=\varnothing$. Let $\{D_j\}_{j\in\mathcal{J}}$ be the set of irreducible components with non-trivial (and hence cyclic) stabilizers $H_j\subset G$, considered up to conjugation in $G$. For each $j\in\mathcal{J}$, the elements of $G$ which do not move $D_j$ to an other component of $\Delta$, give rise to a subgroup $R_j\subset\Centralizer_G(H_j)/H_j$. Consider the subset $\mathcal{I}\subset\mathcal{J}$ corresponding to those divisors, together with the respective $R_j$-action, that cannot be obtained via equivariant blow-ups of any standard model of any $G$-variety (such divisors were called \emph{incompressible}). Finally, let $b_j$ be the character of $H_j$ in the normal bundle to $D_j$. The formal sum
\begin{equation}\label{eq: Burnside}
[X\actsfromright G]=\sum_{i\in\mathcal{I}}\left (H_i,R_i\actsfromleft\kk(D_i),b_i\right ),
\end{equation}
viewed as an element of $\Burn_n(G)$, turns out to be a well-defined $G$-birational invariant \cite[Theorem 5.1]{KreschTschinkelBurnsideVolume}. 

There exist different versions of symbols groups and corresponding Burnside-type obstructions, see e.g. \cite[Section 3]{TschinkelYangZhangLinear} for an overview. The first applications of this formalism yield non-linearizable cyclic actions on certain cubic fourfolds \cite[6]{HassettKreschTschinkel}; a new proof \cite[7.6]{HassettKreschTschinkel} of non-linearizability of minimal $\Dih_6$-action on the sextic del Pezzo surface (this is an Iskovskikh's example \ref{ex: dP6 Iskovskikh's example}); non-conjugacy\footnote{Note that this cannot be approached via the Bogomolov--Prokhorov invariant, which vanishes for linear actions.} of certain intransitive and imprimitive subgroups of $\PGL_3(\CC)$ and $\PGL_4(\CC)$ in $\Cr_2(\CC)$ and $\Cr_3(\CC)$, respectively (by showing that the corresponding actions give different classes in $\Burn_2$ and $\Burn_3$), see \cite[Sections 10-11]{KreschTschinkelRepresentationTheory} and \cite[Sections 7-8]{TschinkelYangZhangLinear}. Furthermore, there are examples of non-linearizable actions on some 3-dimensional quadrics \cite[Section 9]{TschinkelYangZhangLinear}, and classification (obtained through combinations of various techniques) of non-linearizable actions on some prominent threefolds, such as the Segre cubic, the Burkhardt quartic, and some singular cubic threefolds \cite{CheltsovTschinkelZhang,CheltsovTschinkelZhangSingularCubics1,CheltsovTschinkelZhangSingularCubics2}.

\vspace{0.3cm}

Some other equivariant birational invariants were recently proposed by L. Esser (the dual complex \cite{EsserDualComplex}), T. Ciurca, S. Tanimoto and Yu. Tschinkel (an equivariant version of the formalism of intermediate Jacobian torsor obstructions \cite{CiurcaTanimotoTschinkel,CheltsovTschinkelZhangSingularCubics1}), and by J.~Blanc, I.~ Cheltsov, A.~Duncan and Yu.~Prokhorov (the Amitsur subgroup \cite{BlancCheltsovDuncanProkhorov}). We refer the reader to these works for details.

\subsection{Classification of linearizable actions} In this paper, we use an extremely powerful method from birational geometry known as the \emph{Sarkisov program}, which, in principle, has no limitations and allows to answer the question of the birationality of two Mori fibre spaces, and, in particular, their (equivariant) linearizability. The advantage of this method is that, in dimension 2, it is highly explicit (algorithmic); its application in higher dimensions is much more technically involved. In any case, it has enabled us to prove that the ``majority'' of subgroups of the plane Cremona group are non-linearizable, while the linearizable ones belong to the following compact list. We use the standard language of $G$-varieties, which is reminded in Section \ref{sec: G-surfaces}, to formulate our result.

\begin{maintheorem*}
	Let $\kk$ be an algebraically closed field of characteristic zero, and $G\subset\Cr_2(\kk)$ be a finite subgroup. Consider a regularization of $G$ on a two-dimensional $G$-Mori fibre space~$S$ over the base $B$. Then $G$ is \emph{linearizable} if and only if the pair $(S,G)$ is one of the following\footnote{We refer to Notations \ref{notation: group theory} for the group-theoretic notations.}:

\begingroup
\renewcommand*{\arraystretch}{1.4}
\begin{longtable}{|c|l|l|l|}
	\hline
	$K_S^2$ & Surface $S$ & Group $G\subset\Aut(S)$ & Reference  \\ \hline
	
	\rowcolor{light-green}
	\multicolumn{4}{|c|}{$G$-conic bundles $($over $B\simeq\PP^1$$)$} \\ \hline
	
	$K_S^2=8$ & A Hirzebruch surface $\FF_n$ with $n$ odd & --- any & Theorem \ref{thm: hirzebruch answer} \\ \hline
	$K_S^2=8$ & A Hirzebruch surface $\FF_n$ with $n>0$ even & \makecell[l]{--- acts cyclically on $B$\\--- acts as $\Dih_{2m+1}$ on $B$} & Theorem \ref{thm: hirzebruch answer} \\ \hline
	
	$K_S^2=8$ & The quadric surface $\FF_0\simeq\PP^1\times\PP^1$ & \makecell[l]{--- $\Cyc_n\times_Q\Cyc_m$\\ --- $\Cyc_n\times_Q\Dih_{2m+1}$\\ --- $\Dih_{2n+1}\times_Q\Dih_{2m+1}$\\
	is dihedral} &  Theorem \ref{thm: P1xP1 Pic=2 linearization} \\ \hline
	
	\rowcolor{light-green}
	\multicolumn{4}{|c|}{$G$-del Pezzo surfaces $($over $B=\pt$$)$} \\ \hline
		
	$K_S^2=5$  & The unique quintic del Pezzo surface & \makecell[l]{--- $\Cyc_5$\\--- $\Dih_5$} & Proposition \ref{dP5: D5 is linearizable} \\ \hline
	
	$K_S^2=6$  & The unique sextic del Pezzo surface & \makecell[l]{--- $\Cyc_6$\\ --- $\Sym_3$} & Proposition \ref{prop: dP6 criterion} \\ \hline
	
	$K_S^2=8$  & The quadric surface $\PP^1\times\PP^1$ & \makecell[l]{--- $(\Cyc_n\times_Q\Cyc_n)_\bullet\Cyc_2$\\ } &  Proposition \ref{prop: dP8 Pic=1 cyclic} \\ \hline
	
	$K_S^2=9$  & The projective plane $\PP^2$ & --- Blichfeldt's list &  Section \ref{subsec: PGL3} \\ \hline
\end{longtable}
\endgroup	
\end{maintheorem*}

This paper is organized as follows. In Section \ref{sec: G-surfaces}, we recall some key facts about rational $G$-surfaces and their minimal models, following the classical works of V. Iskovskikh. We demonstrate how the initial problem reduces to the linearization of finite groups acting on some very specific models. The main technical tool of the paper --- the Sarkisov program --- is also introduced. This paper employs a significant amount of (elementary) finite group theory, so relevant facts are isolated in Section \ref{sec: group theory}. In Section \ref{sec: del Pezzo}, we derive part of the main theorem related to $G$-del Pezzo surfaces. Although Section \ref{sec: quadrics} is not used later on, it aims to fill a gap in the literature by explicitly describing, in matrix terms, finite groups acting on smooth two-dimensional quadrics (equivalently, finite subgroups of the projective orthogonal group $\PO(4)$). In Section \ref{sec: Hirzebruch}, we examine the linearizability of finite groups acting on Hirzebruch surfaces, which finishes the proof of our main result. Finally, in the Appendix we provide the supporting Magma code for Section \ref{sec: quadrics} (note that this code is for the reader's convenience only and essentially is not used in any proof).

\vspace{0.3cm}

\textbf{Acknowledgement.}
The authors thank Ivan Cheltsov for helpful discussions. The third author is grateful to the CNRS for its support through the grant PEPS JC/JC 2024.
	
\section{Rational $G$-surfaces}\label{sec: G-surfaces}

Throughout the paper, we work over an algebraically closed field $\kk$ of characteristic zero. We first recall the classical approach to classification of finite subgroups in $\Cr_2(\kk)$; we refer to \cite[Section 3]{DolgachevIskovskikh} for more details and proofs.

\subsection{$G$-Minimal model program}\label{subsec: G-MMP} 

Let $G$ be a finite group. A \emph{$G$-surface} is a triple $(S,G,\iota)$, where $S$ is a smooth projective surface over $\kk$ and $\iota\colon G\hookrightarrow \Aut(S)$ is a monomorphism. A $G$-{\it morphism} of $G$-surfaces $(S_1,G,\iota_1)\to (S_2,G,\iota_2)$ is a morphism $f\colon S_1\to S_2$ such that $f\circ\iota_1(G)=\iota_2(G)\circ f$. Similarly, one defines $G$-rational maps and $G$-birational maps.

Let $X$ be an algebraic variety over $\kk$ and $G\subset \Bir(X)$ be a finite group acting birationally and faithfully on $X$. Then one can resolve indeterminacies of the action of $G$ by the following theorem.

\begin{theorem}[{e.g. \cite[Theorem 1.4]{deFernexEin}}]\label{theorem: a_indeterminacies}
	Let $X$ be an algebraic variety and $G\subset \Bir(X)$ be a finite group acting birationally and faithfully on $X$. Then there exists a $G$-equivariant birational map $\phi \colon Y \dashrightarrow X$, where $Y$ is a smooth algebraic variety acted on biregularly by $G$.
\end{theorem}

In this paper we are interested in the case $X=\PP^2$, so the $G$-surface $(Y,G,\iota)$ is rational. Hence, after running the $G$-Minimal model program we find that $(Y,G,\iota)$ is $G$-birational to one of the following \emph{$G$-Mori fibre spaces}:
\begin{enumerate}
	\item[(i)] a \emph{$G$-del Pezzo surface}, i.e. a  $G$-surface $(S,G,\iota)$, such that $S$ is a del Pezzo surface and $\Pic(S)^G\simeq \ZZ$;
	\item[(ii)] a \emph{$G$-conic bundle}, i.e. a conic bundle $\pi\colon S\to \PP^1$ such that, both $S$ and $\PP^1$ are acted on by $G$, $-K_S$ is $\pi$-ample and $\Pic(S)^G\simeq\ZZ^2$. 
\end{enumerate}

Conversely, given a rational $G$-surface $S$, any birational map $\psi\colon S\dasharrow\PP^2$ yields an injective homomorphism 
\[
i_\psi\colon G\hookrightarrow\Cr_2(\kk),\ \ g\mapsto \psi\circ g\circ \psi^{-1}.
\]
We say that $G$ (or the $G$-surface $S$) is \emph{linearizable} if there is a birational map $\psi\colon S\dasharrow\PP^2$ such that $i_{\psi}(G)\subset\Aut(\PP^2)\simeq\PGL_3(\kk)$. If $S'$ is another rational $G$-surface with a birational map $\psi'\colon S'\dasharrow\PP^2$, then the subgroups $i_\psi(G)$ and $i_{\psi'}(G)$ are conjugate if and only if $G$-surfaces $S$ and $S'$ are $G$-birationally equivalent. In other words, a birational equivalence class of $G$-surfaces defines a conjugacy class of subgroups of $\Cr_2(\kk)$ isomorphic to $G$. To sum up, there is a natural bijection between the conjugacy classes of finite subgroups $G\subset\Cr_2(\kk)$ and $G$-birational equivalence classes of rational $G$-Mori fibre spaces of dimension 2.
	
\subsection{Sarkisov program and birational rigidity}\label{subsec: Sarkisov theory} According to the equivariant version of the \emph{Sarkisov program} (see e.g. \cite[Section 7]{DolgachevIskovskikh}), every $G$-birational map between two $G$-Mori fibre spaces can be decomposed into a sequence of $G$-isomorphisms and some ``elementary'' $G$-birational maps, called \emph{Sarkisov $G$-links}. These links come in four types.

\vspace{0.3cm}

{\bf Type I}.
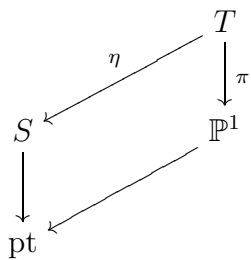
\[
\xymatrix{
	&& T\ar@{->}[dll]_{\eta}\ar@{->}[d]^{\pi}\\
	S\ar[d]&& \PP^1\ar[dll]\\
	{\rm pt} &&
}\]
where $S$ is a $G$-del Pezzo surface, $\eta$ is the blow-up of a $G$-orbit on $S$, and $\pi\colon T\to\PP^1$ is a $G$-conic bundle.

\vspace{0.3cm}

{\bf Type II}.
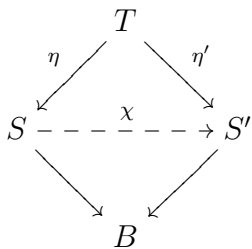
\[\xymatrix{
	&T\ar@{->}[dl]_{\eta}\ar@{->}[dr]^{\eta'}&\\
	S\ar@{-->}[rr]^{\chi}\ar[dr]&& S'\ar[dl]\\
	& {B} &}
\]
where $\eta$ and $\eta'$ are blow-ups of $G$-orbits on $S$ and $S'$, of lengths $d$ and $d'$, correspondently. In this case $S$ and $S'$ are $G$-del Pezzo surfaces if $B=\pt$, or $S$ and $S'$ are $G$-conic bundles if $B\simeq \PP^1$.

\vspace{0.3cm}

{\bf Type III}.
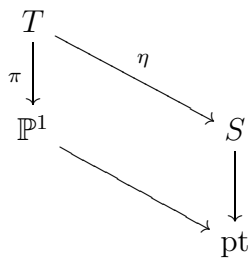
\[
\xymatrix{
	T\ar@{->}[drr]^{\eta}\ar@{->}[d]_{\pi} && \\
	\PP^1\ar[drr] && S\ar[d]\\
	&& {\rm pt}
}
\]
This link is the inverse to the link of type I.

{\bf Type IV}.
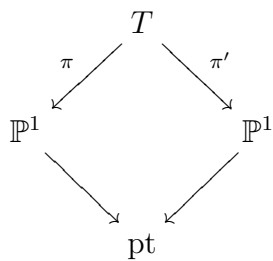
\[
\xymatrix{
	&T\ar@{->}[dl]_{\pi}\ar@{->}[dr]^{\pi'}&\\
	\PP^1\ar[dr]&& \PP^1\ar[dl]\\
	& {\rm pt} &}
\]
This link is the choice of a conic bundle structure on a $G$-conic bundle $T$, which has exactly two such structures. Note that in general such link is not represented by a biregular automorphism of $T$, which exchanges $\pi$ and $\pi'$.

The complete classification of Sarkisov $G$-links between 2-dimensional $G$-Mori fibre spaces was obtained by V. Iskovskikh, see \cite[Theorem 2.6]{Isk1996} and \cite[Propositions 7.12, 7.13]{DolgachevIskovskikh}. We will use extensively this classification in what follows. 

\begin{convention*}
	From now on, when talking about Sarkisov links of some specific types, we always use the notation from the diagrams above. For example, for a Sarkisov link of type ~$\II$ starting at a $G$-del Pezzo surface $S$, $\eta$ always denote the blow-up of $S$.
\end{convention*}

Some $G$-del Pezzo surfaces and $G$-conic bundles have essentially unique structure of a $G$-Mori fibre space. To be more precise, one introduces the following notions.

\begin{definition}
	A $G$-del Pezzo surface $S$ is called \emph{$G$-birationally rigid} if for any $G$-birational map $\chi\colon S\dashrightarrow S'$ to a $G$-Mori fibre space $S'$, the surfaces $S$ and $S'$ are $G$-isomorphic (not necessarily via $\chi$). Furthermore, $S$ is called \emph{$G$-birationally superrigid}, if every such map $\chi$ is already $G$-isomorphism. 
	
	Similarly, a $G$-conic bundle $\pi\colon S\to\PP^1$ is \emph{$G$-birationally rigid}, if for any $G$-birational map $\chi\colon S\dashrightarrow S'$ to the total space of another $G$-conic bundle $\pi'\colon S'\to\PP^1$, the latter is square $G$-birational to the former (not necessarily via $\chi$), i.e. there is a diagram of $G$-maps
	\[
	\xymatrix{
		S\ar@{-->}[rr]^{\varphi}\ar[d]_{\pi} && S'\ar[d]^{\pi'}\\
		\PP^1\ar@{->}[rr]^{\delta} && \PP^1,		
		}
	\]
	where $\delta$ is an isomorphism, and $\varphi$ is birational. If this diagram holds already for $\varphi=\chi$, then $\pi\colon S\to\PP^1$ is called \emph{$G$-birationally superrigid}. 
\end{definition}

Here are two classic examples of $G$-birationally (super)rigid Mori fibre spaces.

\begin{theorem}[Manin--Segre]\label{theorem: Manin-Segre}
	Let $S$ be a $G$-del Pezzo surface. If $K_S^2\leqslant 3$, then $S$ is $G$-birationally rigid. If $K_S^2=1$, then $S$ is $G$-birationally superrigid.
\end{theorem}

A modern proof of this theorem follows from the Noether--Fano method (or, formally, from the classification of Sarkisov $G$-links \cite[Theorem 2.6]{Isk1996}). 

The birational self-maps of $G$-conic bundles with at least 8 singular fibres are also easy to describe: they admit only fibrewise birational transformations.

\begin{theorem}[{\cite[Theorem 1.6]{IskovskikhCB}, \cite[Theorem 2.10]{CheltsovMangolteYasinskyZimmermann}}]\label{theorem: superrigid conic bundles}
	Let $\pi\colon S\to\PP^1$ be a $G$-conic bundle such that $K_S^2\leqslant 0$. Then the~following two assertions hold:
	\begin{itemize}
		\item[(i)] $S$ is not $G$-birational to a~smooth (weak) del Pezzo surface;
		\item[(ii)] Moreover, such $S$ is a $G$-birationally superrigid $G$-conic bundle.
	\end{itemize}
\end{theorem} 

\subsection{$G$-minimal surfaces}

So far we have reduced the problem of classification of finite subgroups in $\Cr_2(\kk)$ to classification of $G$-Mori fibre spaces up to $G$-birational equivalence. Note that $G$-conic bundles $\pi\colon S\to B$ do not have to be $G$-minimal in the absolute sense, i.e. it is not necessarily true that evey $G$-birational morphism $S\to T$ is $G$-isomorphism (a trivial example is the blow-up $\FF_1\to\PP^2$ of a $G$-fixed point on $\PP^2$). However, we have precise description of all such cases. All results below are essentially due to V. Iskovskikh. 

\begin{theorem}[{\cite[Theorems 4 and 5]{Iskovskikh_minimal}}]\label{theorem: minimal conic bundles}
	Let $\pi\colon S\to\PP^1$ be a $G$-conic bundle.
	\begin{enumerate}
		\item\label{minmodels 1} If $S$ is not $G$-minimal, then $S$ is a del Pezzo surface $($in particular, $K_S^2\geqslant 1$$)$.
		\item\label{minmodels 2} Assume that $1\leqslant K_S^2\leqslant 8$. Then $S$ is $G$-minimal if and only if $K_S^2\in\{1,2,4,8\}$.
		\item\label{minmodels 3} If $K_S^2\in\{1,2,4\}$ then $S$ is a del Pezzo surface if and only if there are exactly two $G$-conic bundle structures on $S$.
	\end{enumerate}
\end{theorem}

\begin{lemma}[see {\cite[Theorem 5]{Iskovskikh_minimal}}]
	\label{lem:cbdeg7}
	Let $S$ be a rational $G$-surface with $K_S^2=7$. Then $S$ is not a $G$-Mori fibre space. In other words, $S$ is neither a $G$-del Pezzo surface, nor a $G$-conic bundle.
\end{lemma}

For a $G$-conic bundle $\pi\colon S\to\PP^1$, Noether's formula implies that $K_S^2=8-c$, where $c$ is the number of singular fibres of $\pi$. In particular, $K_S^2\leqslant 8$. Besides, Theorem~\ref{theorem: superrigid conic bundles} says that $G$-conic bundles with $c\geqslant 8$ are $G$-birationally superrigid and hence are not $G$-birational to~$\PP^2$. For our purposes we may assume that $S$ is also $G$-minimal (in the absolute sense), hence  $K_S^2\in\{ 1,2,4,8\}$ by Theorem \ref{theorem: minimal conic bundles}. 

\begin{proposition}
	Let $\pi\colon S\to\PP^1$ be a $G$-conic bundle. If $K_S^2\in\{1,2,4\}$, then $G$ is not linearizable.
\end{proposition}
\begin{proof}
	Indeed, $S$ is $G$-minimal by Theorem \ref{theorem: minimal conic bundles}. In particular, it does not admit Sarkisov links of type $\III$, i.e. contractions to $G$-del Pezzo surfaces. Hence, all Sarkisov $G$-links starting from $S$ are of type $\II$ (elementary transformations) or $\IV$ and do not change $K_S^2$. Therefore, $G$ is not linearizable.
\end{proof}

\textbf{Summary.} So, it remains to investigate the linearizability in the following two cases:
\begin{enumerate}
	\item[---] $S$ is a $G$-del Pezzo surface, where we can assume that $K_S^2\geqslant 4$ by Theorem \ref{theorem: Manin-Segre}. If $K_S^2=8$, we can assume that $S$ is not the blow-up of $\PP^2$ in a $G$-fixed point, because such $S$ is not $G$-minimal; for the same reason, we skip the case $K_S^2=7$, see Lemma~\ref{lem:cbdeg7}.
	\item[---] $S$ is a $G$-conic bundle $\pi\colon S\to\PP^1$ with no singular fibres, i.e. $S$ is a Hirzebruch surface $\FF_n$ acted on by a finite group $G$.
\end{enumerate}

\section{Group theory}\label{sec: group theory}

This Section is entirely devoted to auxiliary results from the theory of finite groups. These results are either elementary (nevertheless, we provide proofs for the reader's convenience) or pertain to the classical representation theory.

\begin{notation}\label{notation: group theory}
	Throughout our paper, we use the following standard notations:
	\begin{enumerate}
		\item[---] $\omega_n=\exp(2\pi i/n)$, where $n\in\mathbb{N}$;
		\item[---] $\Cyc_n$ is a cyclic group of order $n$;
		\item[---] $\Klein_4\simeq\Cyc_2\times\Cyc_2$ is Klein's Vierergruppe;
		\item[---] $\Quat_8$ is the quaternion group;
		\item[---] $\Dih_n$ is the dihedral group of order $2n$;
		\item[---] $\Sym_n$ is the permutation group of degree $n$;
		\item[---] $\Alt_n$ is the alternating group of degree $n$;
		\item[---] $\Hol(G)=G\rtimes\Aut(G)$ is the holomorph of a group $G$;
		\item[---] $\Frobenius_5=\Hol(\Cyc_5)$ is the Frobenius group of order 20;
		\item[---] $\Dih(A)$ is the generalized dihedral group over an abelian group $A$;
		\item[---] $A_\bullet B$ is an extension (not necessarily split) of $B$ with a normal subgroup $A$;
		\item[---] $A\times_Q B$ is a fibred product of $A$ and $B$ over their common homomorphic image $Q$; see Section \ref{subsec: Goursat} for more details;
		\item[---] $G\wr\Sym_n$ is the wreath product, i.e. the semi-direct product $G^n\rtimes\Sym_n$, where $\Sym_n$ acts on $G^n$ by permuting the factors;
		\item[---] $\Centralizer_G(H)$ is the centralizer of $H\subset G$ in $G$;
		\item[---] $\Norm_G(H)$ is the normalizer of $H\subset G$ in $G$;
		\item[---] $\Center(G)$ is the centre of a group $G$;
		\item[---] $\Inn(G)$ is the inner automorphism group of a group $G$. 
	\end{enumerate}
\end{notation}

\subsection{Klein's classification}\label{subsec: PGL2}

First, we recall the following classical result of Felix Klein.

\begin{proposition}[{\cite{KleinIcosahedron}}]\label{prop: Klein}
	If $\kk$ is an algebraically closed field of characteristic zero, then every finite subgroup of $\PGL_2(\kk)$ is isomorphic to $\Cyc_n$, $\Dih_n$ $($where $n\geqslant 1$$)$, $\Alt_4$, $\Sym_4$ or $\Alt_5$. Moreover, there is only one conjugacy class for each of these groups. 
\end{proposition}

In what follows, we will also need the following more general result which describes algebraic subgroups of $\PGL_2(\kk)$.

\begin{theorem}[{\cite[p. 31]{Kaplansky}, \cite[Theorem 1]{NguyenVanDerPutTop}}]\label{thm: algebraic subgrops of PGL2}
	Up to conjugation, every algebraic subgroup of $\PGL_2(\kk)$ is one of the following:
	\begin{enumerate}
		\item $\PGL_2(\kk)$.
		\item A finite subgroup from the Klein's list: $\Cyc_n,\Dih_n,\Alt_4,\Sym_4,\Alt_5$.
		\item The image under the natural projection $\pi\colon\SL_2(\kk)\twoheadrightarrow\PSL_2(\kk)\simeq\PGL_2(\kk)$ of the Borel (i.e. maximal solvable) subgroup
		\[
		{\mathrm B}=\left\{
		\begin{bmatrix}
			a & b\\
			0 & a^{-1}
		\end{bmatrix}\colon a\in\kk^*,b\in\kk
		\right\}\subset\SL_2(\kk).
		\]
		\item The image under the natural projection $\pi\colon\SL_2(\kk)\twoheadrightarrow\PSL_2(\kk)\simeq\PGL_2(\kk)$ of the infinite dihedral subgroup
		\[
		\Dih_\infty=\left\{
		\begin{bmatrix}
			a & 0\\
			0 & a^{-1}
		\end{bmatrix}\colon a\in\kk^*
		\right\}\cup 
		\left\{
		\begin{bmatrix}
			0 & -b\\
			b^{-1} & 0
		\end{bmatrix}\colon b\in\kk^*
		\right\}.
		\]
	\end{enumerate}
\end{theorem}

\begin{remark}\label{rem: structure of algebraic subgroups of PGL2}
	In the case (3), let $\overline{\mathrm B}=\pi(\mathrm{B})$, and consider the unipotent (abelian) subgroup
	\[
	\overline{\mathrm U}=\left\{
	\overline{
	\begin{bmatrix}
		1 & b\\
		0 & 1
	\end{bmatrix}}\colon b\in\kk
	\right\}
	\subset
	\overline{\mathrm B}=\left\{
	\overline{
	\begin{bmatrix}
		a & b\\
		0 & 1
	\end{bmatrix}}\colon a\in\kk^*,b\in\kk
	\right\}
	\]
	The elements of $\overline{\mathrm{B}}$ can be viewed as affine automorphisms of $\PP^1$, locally given by $z\mapsto az+b$. The map which sends every such automorphism to $a\in\kk^*$ is a well-defined group homomorphism, which induces a short exact sequence
	\begin{equation}\label{eq: ses for Borel and Unipotent subgroups}\begin{tikzcd}
		1 
		\ar{r}
		& 
		\overline{\mathrm U}
		\ar{r}
		& 
		\overline{\mathrm B}
		\ar{r}
		& 
		\kk^*
		\ar{r}
		& 
		1.
	\end{tikzcd}
	\end{equation}
	In the case (4), letting $\overline{\Dih}_\infty=\pi(\Dih_\infty)$, we observe that the diagonal matrices $\diag\{a,a^{-1}\}$, where $a\in\kk^*$, constitute the connected component of the identity $\Dih_\infty^\circ$ and are mapped  by $\pi$ to a group isomorphic to $\kk^*$. The other matrices are mapped to involutions. Furthermore, one has $\overline{\Dih}_\infty\simeq\kk^*\rtimes\Cyc_2$, where the action is by inversion: $t\mapsto t^{-1}$, i.e. $\overline{\Dih}_\infty\simeq\Hol(\kk^*)$. 
\end{remark}

Let us fix the following presentations \cite[\S 19]{Huppert}:

\begin{itemize}
	\item[---] $\Cyc_n\simeq \langle r\ |\ r^n=\id\rangle$.
	\item[---] $\Dih_n\simeq \langle r,b\ |\ r^n=b^2=(rb)^2=\id\rangle$.
	\item[---] $\Alt_4\simeq\langle a,b,c\ |\ a^2=b^2=c^3=\id,\ cac^{-1}=ab=ba,\ cbc^{-1}=a \rangle $. Note that $\langle a,b\rangle\simeq\Klein_4$ is the derived subgroup of $\Alt_4$ and one has $\Alt_4\simeq \langle a,b\rangle\rtimes\langle c\rangle$.
	\item[---] $\Sym_4\simeq\langle a,b,c,d\ |\ a^2=b^2=c^3=d^2=\id,\ cac^{-1}=dad=ab=ba,\ cbc^{-1}=a,\ bd=db,\ dcd=c^{-1} \rangle $. Here, one finds a unique copy of $\Alt_4\simeq \langle a,b,c\rangle $ inside, which is the derived subgroup of $\Sym_4$, and one has $\Sym_4\simeq\langle a,b,c\rangle\rtimes\langle d\rangle$.
	\item[---] $\Alt_5\simeq\langle e,f\ |\  e^5=f^2=(ef)^3=\id\rangle$.
\end{itemize}

\begin{notation}\label{notation: matrices}
Let us fix the following identifications (see e.g. \cite[Theorem C]{Faber}):
\rowcolors{1}{}{}
\[
r\mapsto \R_n=
\begin{bmatrix}
	1 & 0\\
	0 & \omega_n
\end{bmatrix},\ 
a\mapsto \A=
\begin{bmatrix}
	1 & 0\\
	0 & -1
\end{bmatrix},\ 
b\mapsto \B=
\begin{bmatrix}
	0 & 1\\
	1 & 0
\end{bmatrix},\ 
c\mapsto \C=
\begin{bmatrix}
	i & -i\\
	1 & 1
\end{bmatrix},\ 
\]
\[
d\mapsto 
\D=\begin{bmatrix}
	1 & -i\\
	i & -1
\end{bmatrix},\ 
e\mapsto 
\E=\begin{bmatrix}
	\omega_5 & 0\\
	0 & 1
\end{bmatrix},\ 
f\mapsto 
\F=
\begin{bmatrix}
	1 & 1-\omega_5-\omega_5^{-1}\\
	1 & -1
\end{bmatrix}.
\]
Then these identifications define embeddings of $\Cyc_n$, $\Dih_n$, $\Alt_4$, $\Sym_4$ and $\Alt_5$ into $\PGL_2(\kk)$.
\end{notation}

The following description of orbits on $\PP^1$ is classical:

\begin{proposition}[{\cite[4.4]{SpringerInvariantTheory}}]\label{prop: Klein orbits}
	Let $G$ be a finite subgroup of $\PGL_2(\kk)$. One has the following.
	\begin{enumerate}
		\item If $G\simeq\Cyc_n$ then it has $2$ fixed points on $\PP^1$, namely $[1:0]$ and $[0:1]$, and any other point generates an orbit of length $n$.
		\item If $G\simeq\Dih_n$, then it has one orbit of length $2$, namely $[1:0]$ and $[0:1]$, and one orbit of length $n$ generated by $[1:1]$. Any other point generates an orbit of length $2n$.
		\item If $G\simeq\Alt_4$ then it has two orbits of length $4$ and one orbit of length $6$. All other orbits are of length $12$.
		\item If $G\simeq\Sym_4$ then it has one orbit of length $6$, one orbit	of length $8$ and one orbit of length $12$. All other orbits are of length $24$.
		\item If $G\simeq\Alt_5$ then it has one orbit of length $12$, one orbit of length $20$, and one orbit of length $30$. All other orbits are of length $60$.
	\end{enumerate}
\end{proposition}

\subsection{Blichfeldt's classification}\label{subsec: PGL3}

Since we are interested in linearization of finite subgroups of $\Cr_2(\kk)$, let us first recall the classification of finite subgroups of $\PGL_3(\kk)$, which is essentially due to H. Blichfeldt. 

\begin{definition}[{\cite{Blichfeldt}}]\label{def: Blichfeldt}
	We call a subgroup $\iota\colon G\hookrightarrow\GL_n(\kk)$ {\it intransitive} if the representation $\iota$ is reducible, and {\it transitive} otherwise. Further, a transitive group $G$ is called {\it imprimitive} if there is a decomposition $\kk^n=\bigoplus_{i=1}^mV_i$ into a direct sum of subspaces and $G$ transitively acts on the set $\{V_i\}$. A transitive group $G$ is called {\it primitive} if there is no such decomposition. Finally, we say that $G\subset\PGL_n(\kk)$ is (in)transitive or (im)primitive if its lift to $\GL_n(\kk)$ is such a group.
\end{definition}

So, intransitive subgroups of $\PGL_3(\kk)$ come in two types:
\begin{enumerate}
	\item[$I_{1}$] The representation of $G$ in $\GL_3(\kk)$ is a direct sum of three 1-dimensional representations. In other words, $G$ fixes 3 non-collinear points on $\PP^2$ and hence $G\simeq\Cyc_n\times\Cyc_m$, where $n,m\geqslant 1$, is a diagonal abelian group. 
	\item[$I_2$] The representation of $G$ in $\GL_3(\kk)$ is a direct sum of 1-dimensional and 2-dimensional representations. In other words, $G$ fixes a point $p\in\PP^2$, and hence there is an embedding $G\hookrightarrow \GL(T_p\PP^2)\simeq\GL_2(\kk)$. Obviously, every finite subgroup of $\GL_2(\kk)$ gives rise to an intransitive subgroup of $\PGL_3(\kk)$.  
\end{enumerate}

Assume that $G$ is imprimitive. Then $\kk^3=V_1\oplus V_2\oplus V_3$, where $V_1\simeq V_2\simeq V_3\simeq\kk$ and $G$ acts transitively on the set $\{V_1,V_2,V_3\}$. Hence $G$ fits into the short exact sequence
\[\begin{tikzcd}
	1 
	\ar{r}
	& 
	A
	\ar{r}
	& 
	G
	\ar{r}
	& 
	B
	\ar{r}
	& 
	1
\end{tikzcd}
\]
where $B\simeq\Cyc_3$ or $B\simeq\Sym_3$, and $A\subset\PGL_3(\kk)$ is an intransitive subgroup of type $I_1$.

Finally, there are 6 primitive subgroups of $\PGL_3(\kk)$.

\begingroup
\renewcommand*{\arraystretch}{1.2}
\begin{longtable}{|c|c|c|c|}
	\hline
	GAP & Order & Isomorphism class & Comments \\ \hline
	
	\rowcolor{light-green}
	\multicolumn{4}{|c|}{Primitive subgroups} \\ \hline

	$[36,9]$ & 36 & $(\Cyc_3\times \Cyc_3)\rtimes\Cyc_4$ & subgroup of the Hessian group  \\ \hline
	\rowcolor{light-green}
	$[72,41]$ & 72 & $(\Cyc_3\times \Cyc_3)\rtimes\Quat_8$ & subgroup of the Hessian group \\ \hline
	$[216, 153]$ & 216 & $(\Cyc_3\times \Cyc_3)\rtimes\SL_2(\FF_3)$ & the Hessian group \\ \hline
	\rowcolor{light-green}
	$[60,5]$ & 60 & $\Alt_5$ & the simple icosahedral group \\ \hline
	$[168,42]$ & 168 & $\PSL_2(\FF_7)$ & the simple Klein group \\ \hline
	\rowcolor{light-green}
	$[360,118]$ & 360 & $\Alt_6$ & the simple Valentiner group \\ \hline
\end{longtable}
\endgroup

\subsection{Goursat's lemma}\label{subsec: Goursat}

Let $\Gamma_1$ and $\Gamma_2$ be two finite groups. Finite subgroups of the direct product $\Gamma_1\times\Gamma_2$ can be determined using the classical \emph{Goursat's lemma}. Recall that the \emph{fibre product} of two groups $G_1$ and $G_2$ over a group $Q$ is defined as
\[
G_1\times_{Q} G_2=\{(g_1,g_2)\in G_1\times G_2\colon \alpha(g_1)=\beta(g_2)\},
\] 
where $\alpha\colon G_1\to Q$ and $\beta\colon G_2\to Q$ are surjective homomorphisms. Note that the data defining $G_1\times_Q G_2$ is not only
the groups $G_1$, $G_2$ and $Q$ but also the homomorphisms $\alpha$, $\beta$.

\begin{lemma}[{{Goursat's lemma}, \cite[p. 47]{Goursat}}]\label{lem: Goursat}
	Let $\Gamma_1$ and $\Gamma_2$ be two finite groups. There is a bijective correspondence between subgroups $G\subseteq \Gamma_1\times \Gamma_2$ and $5$-tuples $\{G_1,G_2,H_1,H_2,\varphi\}$, where $G_1$ is a subgroup of $\Gamma_1$, $H_1$ is a normal subgroup of $G_1$, $G_2$ is a subgroup of $\Gamma_2$, $H_2$ is a normal subgroup of $G_2$, and $\varphi\colon G_1/H_1\iso G_2/H_2$ is an isomorphism. Namely, the group corresponding to this $5$-tuple is
	\[
	G=\{(g_1,g_2)\in G_1\times G_2\colon \varphi(g_1H_1)=g_2H_2\}.
	\]
	Conversely, let $G\subseteq\Gamma_1\times \Gamma_2$ be a subgroup. Denote by $p_1\colon \Gamma_1\times \Gamma_2\to \Gamma_1$ and $p_2\colon \Gamma_1\times \Gamma_2\to \Gamma_2$ the natural projections, and set $G_1=p_1(G)$, $G_2=p_2(G)$. Set
	\[
	H_1=\ker p_2|_G=\{(g_1,\id)\in G,\ g_1\in \Gamma_1\},\ \ H_2=\ker p_1|_G=\{(\id,g_2)\in G,\ g_2\in \Gamma_2\},
	\]
	whose images by $p_1$ and $p_2$ define normal subgroups of $G_1$ and $G_2$, respectively $($denoted the same$)$. Let $\pi_1\colon G_1\to G_1/H_1$ and $\pi_2\colon G_2\to G_2/H_2$ be the quotient homomorphisms. The map $\varphi\colon G_1/H_1\to G_2/H_2$, $\varphi(g_1H_1)=g_2H_2$, where $g_2\in \Gamma_2$ is any element such that $(g_1,g_2)\in G$, is an isomorphism. Furthermore, $G=G_1\times_Q G_2$, where $Q=G_1/H_1$, $\alpha=\pi_1$, and $\beta=\varphi^{-1}\circ\pi_2$.
	
\end{lemma}

\begin{lemma}\label{lem: Goursat exact sequence}
	In the above notation, any subgroup $G\subseteq \Gamma_1\times \Gamma_2$ fits into the short exact sequence
	\[\begin{tikzcd}
		1 
		\ar{r}
		& 
		H_1\times H_2
		\ar{r}
		& 
		G
		\ar{r}
		& 
		Q
		\ar{r}
		& 
		1
	\end{tikzcd}
	\]
\end{lemma}
\begin{proof}
	Indeed, the restriction of the homomorphism $\alpha\times\beta\colon G_1\times G_2\to Q\times Q$ to $G$ has the kernel $H_1\times H_2$ and the image is isomorphic to $\{(t,t)\in Q\times Q\}\simeq Q$.
\end{proof}

\begin{remark}\label{Generators of a fibre product}
	 We will often use explicit generators to describe a fibre product. They are given as follows. Let $S$ be a subset of $G_1$ such that $\alpha(S)=Q$. Then the group $G_1\times_QG_2$ is generated by $\ker\alpha\times\ker\beta$, and $\{(g_1,g_2),g_1\in S,\alpha(g_1)=\beta(g_2)\}$.
\end{remark}

\subsection{Actions on sets} We will use several elementary lemmas about group actions on a set.

\begin{lemma}\label{lem: orbits of a direct product}
	Let $A$ and $B$ be two finite groups, and $G\subset A\times B$ be a subgroup such that the natural projection $G\to A$ is surjective. Assume that $A$ acts on a set $X$ and $B$ acts on a set $Y$. Consider the induced action of $G$ on $X\times Y$. Then the length of any orbit of $G$ in $X\times Y$ is divisible by the length of some orbit of $A$ on $X$.
\end{lemma}
\begin{proof}
	Pick a point $(x_0,y_0)\in X\times Y$ and consider its $G$-orbit $\Sigma=\{(ax_0,by_0)\colon (a,b)\in G\}$. Consider the projection $p_X\colon X\times Y\to X$. We claim that the set $\Sigma_X=p_X(\Sigma)\subset X$ is an $A$-orbit. Firstly, it is clearly $A$-invariant: taking any $x\in \Sigma_X$ there exists $y\in Y$ such that $(x,y)\in\Sigma$. By our assumption, for any $a\in A$ there exists $b\in B$ such that $(a,b)\in G$. Then $(ax,by)\in\Sigma$ and therefore $ax\in\Sigma_X$. Secondly, $\Sigma_X$ is clearly a minimal $A$-invariant set. For if there is a proper $A$-invariant subset $\Sigma_X'\subsetneq\Sigma_X$, then taking two elements $x'\in\Sigma'_X$, $x\in\Sigma_X\setminus\Sigma_X'$ and their lifts $(x',y')\in\Sigma$, $(x,y)\in\Sigma$, there exists an element $(a,b)\in G$ such that $(ax',by')=(x,y)$. In particular, $x=ax'\notin\Sigma_X'$; but this contradicts to $\Sigma_X'$ being $A$-invariant.
	
	It remains to show that for each $x_1\in\Sigma_X$, the cardinality of the fibre $p_X^{-1}(x_1)\cap\Sigma$ is the same and equals to the cardinality of $p_X^{-1}(x_0)\cap\Sigma$. Indeed, there exists $a\in A$ such that $ax_0=x_1$. Pick any lift $(a,b)\in G$ and consider the map of finite sets
	\[
	p_X^{-1}(x_0)\cap\Sigma\to p_X^{-1}(x_1)\cap\Sigma,\ (x_0,y)\mapsto (ax_0,by)=(x_1,by),
	\]
	which is obviously well defined and injective. The same holds for the map $p_X^{-1}(x_1)\cap\Sigma\to p_X^{-1}(x_0)\cap\Sigma$ which sends $(x_1,y)$ to $(a^{-1}x_1,b^{-1}y)=(x_0,b^{-1}y)$. This finishes the proof.
\end{proof}

\begin{lemma}\label{lem: orbits of subgroups}
	Let $G$ be a finite group and $H\subset G$ be a normal subgroup of finite index $[G:H]$. Assume that $G$ acts on a set $X$ and let $\Sigma$ be an orbit of $G$. Write $\Sigma=\Sigma_1\sqcup\ldots\sqcup\Sigma_n$ as the disjoint union of $H$-orbits $\Sigma_i$. Then all $\Sigma_i$ have the same length and $n$ divides $[G:H]$. In particular, the length of $\Sigma$ is always divisible by the length of some orbit of $H$. 
\end{lemma}
\begin{proof}
Indeed, the quotient group $G/H$ acts on the space of orbits $\Sigma/H$, and this action is transitive. Therefore, by the Orbit-Stabilizer Theorem, $[G:H]$ is divisible by the cardinality of $\Sigma/H$, i.e. by $n$. Consider any two $H$-orbits $\Sigma_i$ and $\Sigma_j$, let $x_i\in\Sigma_i$ and $x_j\in\Sigma_j$, so that $\Sigma_i=Hx_i$ and $\Sigma_j=Hx_j$. By transitivity of $G$ on $\Sigma$, there exists $g\in G$ such that $gx_i=x_j$. Then the map $\Sigma_i\to\Sigma_j$, $y\mapsto gy$ is a bijection.
\end{proof}

\subsection{Generalized dihedral groups} We remind the following elementary fact.

\begin{lemma}\label{lem: dihedral group subgroups}
	Subgroups of the dihedral group $\Dih_n=\langle r,s\ |\ r^n=s^2=(sr)^2=\id\rangle$ are the following:
	\begin{description}
		\item[\it Cyclic] $\langle r^d\rangle\simeq\Cyc_{n/d}$, and $\langle r^ks\rangle$, where $d$ divides $n$ and $0\leqslant k\leqslant n-1$.
		\item[\it Dihedral] $\langle r^d,r^ks\rangle\simeq\Dih_{n/d}$, where $d<n$ divides $n$, and $0\leqslant k\leqslant d-1$. 
	\end{description}
	All cyclic subgroups $\langle r^d\rangle$ are normal, one has $\Dih_n/\langle r^d\rangle\simeq\Dih_{d}$, and these are all normal subgroups when $n$ is odd. When $n$ is even, there are two additional normal dihedral subgroups of index $2$, namely $\langle r^2,s\rangle$ and $\langle r^2,rs\rangle$.
\end{lemma}

Let $A$ be an abelian group. Recall that the \emph{generalized dihedral group} $\Dih(A)$ is the semi-direct product $A\rtimes\Cyc_2$, where $\Cyc_2$ acts by inverting the elements of $A$. In particular, for $A\simeq\Cyc_n$ one has $\Dih(\Cyc_n)\simeq\Dih_n$.

\begin{lemma}\label{lem: subgroups of Dn x Dm}
	Let $n,m\geqslant 3$ be odd integers. Then one has the following:
	\begin{enumerate}
		\item Any group $G=\Dih_n\times_Q\Dih_m$ is either $\Dih_n\times\Dih_m$, or a generalized dihedral group $\Dih(A)$, where $A$ is a direct product of at most two cyclic groups.
		\item Furthermore, $G\simeq\Dih(A)$ admits a faithful 2-dimensional representation if and only if $A$ is cyclic, so that $G$ is isomorphic to a dihedral group.
	\end{enumerate}
\end{lemma}
\begin{proof}
	Let us fix the presentations $\Dih_n=\langle r_1,s_1\ |\ r_1^n=s_1^2=(s_1r_1)^2=\id\rangle$, $\Dih_m=\langle r_2,s_2\ |\  r_2^m=s_2^2=(s_2r_2)^2=\id\rangle$. We then have the following possibilities for $G$:
	\begin{enumerate}
		\item[(i)] $Q\simeq\id$. In this case, Lemma \ref{lem: Goursat exact sequence} implies that $G\simeq\Dih_n\times\Dih_m$.
		\item[(ii)] $Q\simeq\Cyc_2$. Since $n$ and $m$ are odd, Lemma \ref{lem: Goursat exact sequence} and Lemma \ref{lem: dihedral group subgroups} show that $G$ is an extension of $\Cyc_2$ by $A=\Cyc_n\times\Cyc_m$. Let $\tau\in G$ be an element mapped to the generator of $Q$. It necessarily belongs to $\Dih_n\times\Dih_m\setminus A$ and, being in the fibre product, is of the form $(s_1r_1^k,s_2r_2^l)$. The conjugation by $\tau$ induces an inversion on $A$, hence the claim. 
		\item[(iii)] $Q\simeq\Dih_q$, where $q$ divides both $n$ and $m$. Then the group $G$ fits into the short exact sequence
		\[\begin{tikzcd}
			1 
			\ar{r}
			& 
			\Cyc_{n/q}\times\Cyc_{m/q}
			\ar{r}
			& 
			G
			\ar{r}{\psi}
			& 
			\Dih_q
			\ar{r}
			& 
			1,
		\end{tikzcd}\]
		where the groups $\Cyc_{n/q}$ and $\Cyc_{m/q}$ are generated by $r_1^q$ and $r_2^q$, respectively. Let $A=\psi^{-1}(\Cyc_q)$. This is an index 2 subgroup in $G$ of odd order; therefore, all its elements are of the form $(r_1^k,r_2^l)$. Hence $A$ is a subgroup of $\Cyc_n\times\Cyc_m\subset\Dih_n\times\Dih_m$. As in the previous case, we conclude that $G\simeq\Dih(A)$.
	\end{enumerate}
To prove the second claim, it is enough to notice that $\Center(\Dih(A))=\id$ for any non-trivial $A\subset\Cyc_n\times\Cyc_m$, since $n$ and $m$ are odd. In particular, if $G$ admits a faithful 2-dimensional representation, then $G$ can be embedded in $\PGL_2(\kk)$, hence $A$ is a cyclic group. The converse is obvious.
\end{proof}

\section{$G$-del Pezzo surfaces}\label{sec: del Pezzo}

\subsection{Degree 4} By Theorem \ref{theorem: Manin-Segre}, none of the groups $G$ which act biregularly on a del Pezzo surface $S$ with $\Pic(S)^G\simeq\ZZ$ and $K_S^2\leqslant 3$ is linearizable. Let $S$ be a $G$-del Pezzo surface of degree 4 and $\varphi\colon S\dashrightarrow S'$ be a $G$-birational map to another $G$-del Pezzo surface $S'$. It follows then from the classification of Sarkisov $G$-links that $\varphi$, if not an isomorphism, fits into the commutative diagram of $G$-maps
\begin{equation}\label{eq: sequence of links}
\xymatrix{
	& T\ar[d]\ar[dl]_{\pi}\ar@{-->}[r]^{\chi_1} & T_1\ar@{-->}[r]^{\chi_2}\ar[d] & T_2\ar[d]\ar@{-->}[r]^{\chi_3} & \cdots & \cdots\ar@{-->}[r]^{\chi_{n}} & T'\ar[dr]^{\pi'}\ar[d] & \\
	S & \PP^1\ar@{=}[r] & \PP^1\ar@{=}[r] & \PP^1\ar@{=}[r] && \ar@{=}[r] & \PP^1 & S'		
}
\end{equation}
where $\pi$ is the blow-up of a $G$-fixed point on $S$, which leads to a cubic surface $T$ equipped with a structure of a $G$-conic bundle (so $\pi$ is a link of type $\I$); the maps $\chi_i$ are elementary transformations of $G$-conic bundles (links of type $\II$), and $\pi'$ is the blow-down of a $G$-orbit of $(-1)$-curves, where $S'$ is again a $G$-del Pezzo surface of degree 4. We conclude that none of such groups $G$ is linearizable.

\subsection{Degree 5}\label{subsec: dP5} Let $S$ be a $G$-del Pezzo surface of degree 5. Recall that there is a single isomorphism class of del Pezzo surfaces of degree 5 over $\kk=\overline{\kk}$. Every Sarkisov $G$-link starting from $S$ is of type II, where $\eta$ blows up a $G$-orbit of length $d$, and one of the following holds:
\begin{enumerate}
	\item $S\simeq S'$, $d=4$, $\chi$ is a birational Bertini involution;
	\item $S\simeq S'$, $d=3$, $\chi$ is a birational Geiser involution;
	\item $S'\simeq\PP^1\times\PP^1$, $d=2$;
	\item $S'\simeq\PP^2$, $d=1$.
\end{enumerate}

The following Proposition lists all possible groups $G$.

\begin{proposition}[{\cite[Theorem 6.4]{DolgachevIskovskikh}}]\label{prop: dP5 minimal groups}
	Let $S$ be a del Pezzo surface of degree $5$ and $G\subset\Aut(S)$ be a group such that $\Pic(S)^G\simeq\ZZ$. Then $G$ is isomorphic to one of the following five groups:
	\[
	\Sym_5,\ \ \Alt_5,\ \ \Frobenius_5,\ \ \Dih_{5},\ \ \Cyc_5,
	\]
	where $\Frobenius_5=\langle a,b\ |\ a^5=b^4=1,bab^{-1}=a^3 \rangle $ is the Frobenius group\footnote{GAP ID [20,3]} of order 20.
\end{proposition}

\begin{proposition}[{\cite[Example 6.3]{CheltsovGAFA}, \cite[Theorem B.10]{CheltsovTwoLocalInequalities}}]\label{prop: dP5 S5-rigidity}
	The del Pezzo surface of degree 5 is $\Alt_5$- and $\Sym_5$-birationally superrigid. 
\end{proposition}
\begin{proof}
	None of these two groups has faithful 2-dimensional representations over $\CC$, hence none of them has an orbit of size 1 or 2 on our surface: recall \cite[Lemma 2.4]{BialynickiBirula} that if $X$ is an irreducible algebraic variety and $G\subset\Aut(X)$ is a finite group fixing a point $p\in X$ then the induced linear representation $G\hookrightarrow\GL(T_pX)$ is faithful. Furthermore, these groups have no subgroups of index 3 or 4, hence there are no equivariant birational Bertini or Geiser involutions on (with respect to these groups). The result follows.
\end{proof}

\begin{proposition}\label{dP5: D5 is linearizable}
	Let $S$ be a $G$-del Pezzo surface of degree 5. Then $G$ is linearizable if and only if $G\simeq\Cyc_5$ or $G\simeq\Dih_5$.
\end{proposition}
\begin{proof}
	Proposition \ref{prop: dP5 S5-rigidity} implies that the groups $\Alt_5$ and $\Sym_5$ are not linearizable; in fact, $\Sym_5$ is not a subgroup of $\PGL_3(\CC)$. The Frobenius group $\Frobenius_5$ is not a subgroup of $\PGL_3(\CC)$ either: according to description given in Section \ref{subsec: PGL3}, it is obviously not transitive, and moreover has no faithful 2-dimensional representations (in fact, by \cite[Theorem 1.1]{Wolter}, there exists a unique $\Frobenius_5$-del Pezzo surface which is $\Frobenius_5$-birational to $S$, namely $\PP^1\times\PP^1$). 
	
	We now show that the groups $\Cyc_5$ and $\Dih_5$ are linearizable. Let $G$ be any of these groups. It is enough to construct an explicit $G$-equivariant map $\PP^2\dashrightarrow S$, where $S$ is a $G$-del Pezzo surface of degree 5, because there is a single $G$-isomorphism class of such surfaces. Indeed, there is a single conjugacy class in $\Aut(S)\simeq\Sym_5$ of $\Cyc_5$ and $\Dih_5$, so any isomorphism $S\iso S'$ to the del Pezzo surface $S'$ of degree 5 can be made a $G$-isomorphism, after composing with a suitable automorphism of $S'$. Now, to construct an explicit $G$-birational map $\PP^2\dashrightarrow S$, one can take two linear automorphisms
	\[
	r\colon [x:y:z]\mapsto [x:\omega_5y:\omega_5^{-1}z],\ \ s\colon [x:y:z]\mapsto [x:z:y],
	\] 
	and let $G_1=\langle r\rangle\simeq\Cyc_5$, $G_2=\langle r,s\rangle\simeq\Dih_5$. The $G_1$- and $G_2$-orbit of the point $[1:1:1]$ is a set of 5 points in general position on $\PP^2$, lying on the unique smooth conic $Q\subset\PP^2$. By blowing up these points and contracting the strict transform of $Q$ we get a $G_1$-birational (respectively, $G_2$-birational) map from $\PP^2$ to $S$.
\end{proof}

\begin{remark}\label{rem: dP5 stable linearizability}
	The actions of $\Alt_5$ and $\Sym_5$ are stably linearizable. More precisely, by \cite[Proposition 4.7]{ProkhorovFieldsOfInvariants}, Prokhorov shows that  $S\times\PP^1$ is $\Sym_5$-birational to the Segre cubic threefold $\sum_{i=0}^{5}x_i=\sum_{i=0}^{5}x_i^3=0$, which is obviously acted on by $\Sym_6$. Recall that, up to conjugation, the group $\Sym_6$ has two subgroups isomorphic to $\Sym_5$, given by the standard and
	the non-standard embeddings (which differ by an outer automorphism of $\Sym_6$). In the construction above, $S\times\PP^1$ turns out to be $\Sym_5$-birational to the Segre cubic with a non-standard action of $\Sym_5$, which is linearizable. Recently, B. Hassett and Yu. Tschinkel gave another proof of this fact using the equivariant torsor formalism \cite[\S 8.2]{HassettTschinkelTorsors}.
\end{remark}

\subsection{Degree 6}\label{subsec: dP6} Let $S$ be a del Pezzo surface of degree $K_S^2=6$. Then $S$ is the blow-up $\pi\colon S\to\PP^2$ in three non-collinear points $p_1,p_2,p_3$, which we may assume to be $[1:0:0]$, $[0:1:0]$ and $[0:0:1]$, respectively. The surface $S$ can be given as
\begin{equation}\label{eq: del Pezzo of degree 6}
	\big\{([x_0:x_1:x_2],[y_0:y_1:y_2])\in\PP^2\times\PP^2:\ x_0y_0=x_1y_1=x_2y_2 \big\}.
\end{equation}
The set of $(-1)$-curves on $S$ consists of six curves: the exceptional divisors of blow-up $e_i={\pi}^{-1}(p_i)$ and the strict transforms of the lines $d_{ij}$ passing through $p_i,\ p_j$. In the anticanonical embedding $S\hookrightarrow\PP^6$ these $(-1)$-curves form a hexagon $\Sigma$; the configuration of lines is shown in the diagram (\ref{pic: hexagon}) below.
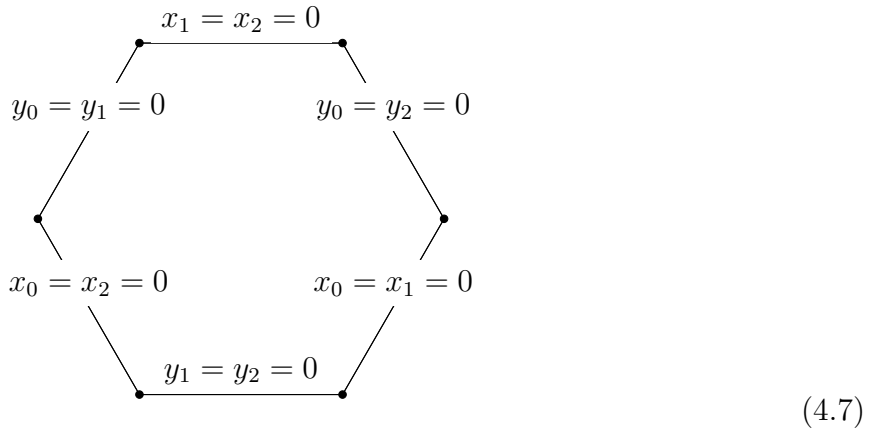
\begin{equation}\label{pic: hexagon}
	\begin{tikzpicture}
		\newdimen\R
		\R=2.7cm
		\draw (0:\R) \foreach \x in {60,120,...,360} {  -- (\x:\R) };
		
		\foreach \x/\l/\p in
		{ 60/{}/above,
			120/{}/above,
			180/{}/left,
			240/{}/below,
			300/{}/below,
			360/{}/right
		}
		\node[inner sep=1pt,circle,draw,fill,label={\p:\l}] at (\x:\R) {};
		
		\foreach \start/\end/\label in
		{ 0/60/$y_0=y_2=0$,
			60/120/$x_1=x_2=0$,
			120/180/$y_0=y_1=0$,
			180/240/$x_0=x_2=0$,
			240/300/$y_1=y_2=0$,
			300/360/$x_0=x_1=0$
		}
		\draw (\start:\R) -- (\end:\R) node[midway, above, fill=white] {\label};
	\end{tikzpicture}
\end{equation}
Note that $\Sigma$ is naturally acted on by $\Aut(S)$, so there is a homomorphism
\[
\Phi\colon \Aut(S)\to \Aut(\Sigma)\simeq\Dih_{6}=\langle r,s\ |\ r^6=s^2=\id,\ srs^{-1}=r^{-1}\rangle,
\]
where $r$ is a rotation by $\pi/3$ and $s$ is a reflection. The kernel of $\Phi$ is the maximal torus $\Torus\subset\PGL_3(\kk)$, isomorphic to $(\kk^*)^3/\kk^*\simeq(\kk^*)^2$, and acts on $S$ by
\begin{equation}\label{eq: del Pezzo 6 torus action}
	(\lambda_0,\lambda_1,\lambda_2)\cdot ([x_0:x_1:x_2],[y_0:y_1:y_2])=([\lambda_0 x_0:\lambda_1x_1:\lambda_2x_2],[\lambda_0^{-1}y_0:\lambda_1^{-1}y_1:\lambda_2^{-1}y_2]).
\end{equation}
The action of $\Torus$ on $S\setminus\Sigma$ is faithful and transitive, and the automorphism group of $\Aut(S)$ fits into the short exact sequence 
\[
\begin{tikzcd}
	1 
	\ar{r}
	& 
	{\Torus} 
	\ar{r}
	& 
	\Aut(S) 
	\ar{r}{\Phi}
	& 
	\Dih_6
	\ar{r}
	& 
	1
\end{tikzcd}
\]
with $\Phi(\Aut(S))\simeq\Dih_6\simeq\Sym_3\times\Cyc_2$, where $\Cyc_2$ is the lift of the standard quadratic involution, acting by
\begin{equation}\label{eq: del Pezzo 6 Cremona}
	\iota\colon ([x_0:x_1:x_2],[y_0:y_1:y_2])\mapsto ([y_0:y_1:y_2],[x_0:x_1:x_2]),
\end{equation}
and $\Sym_3$ acts naturally by permutations on each of the two triples $x_0,x_1,x_2$ and $y_0,y_1,y_2$. In what follows, we will denote
\begin{equation}\label{eq: dP6 order 3 rotation}
	\theta\colon ([x_0:x_1:x_2],[y_0:y_1:y_2])\mapsto ([x_1:x_2:x_0],[y_1:y_2:y_0])
\end{equation}
the generator of $\Cyc_3\subset\Sym_3$. Then $\Dih_6$ is generated by the element $\rho=\theta\circ\iota$ of order 6 (rotation) and the element $\sigma$ of order 2 (reflection), whose actions are given by
\begin{gather}\label{eq: generators of D6}
	\rho\colon ([x_0:x_1:x_2],[y_0:y_1:y_2])\mapsto ([y_1:y_2:y_0],[x_1:x_2:x_0]),\\
	\sigma\colon ([x_0:x_1:x_2],[y_0:y_1:y_2])\mapsto ([y_0:y_2:y_1],[x_0:x_2:x_1]).	
\end{gather}

Every Sarkisov link $\chi\colon S\dashrightarrow S'$ is of type II and starts with blowing up a $G$-orbit of length~$d$. Then one of the following holds:
\begin{enumerate}
	\item $S\simeq S'$, $d=5$, $\chi$ is a birational Bertini involution;
	\item $S\simeq S'$, $d=4$, $\chi$ is a birational Geiser involution;
	\item $d=3$, $K_{S'}^2=6$;
	\item $d=2$, $K_{S'}^2=6$;
	\item $d=1$, $S'\simeq\PP^1\times\PP^1$.
\end{enumerate}

\begin{remark}
	In the cases (3) and (4), the surfaces $S'$ \emph{does not} have to be $G$-isomorphic to~$S$, see \cite[Example 3.9]{YasinskyGBirationalRigidity}.
\end{remark}

\begin{lemma}[{\cite[Lemma 3.7]{YasinskyGBirationalRigidity}}]\label{lem: dP6 minimal groups}
	Let $S$ be a $G$-del Pezzo surface of degree $6$. Then $G$ is of the form
	\[
	T_\bullet\langle r\rangle\simeq T_\bullet \Cyc_6,\ \ \ \ \ \ \ T_\bullet\langle r^2,s\rangle\simeq  T_\bullet\Sym_3,\ \ \ \text{or}\ \ \ \ T_\bullet\langle r,s\rangle\simeq T_\bullet\Dih_6,
	\]
	where $T\simeq\Cyc_n\times\Cyc_m$ is a subgroup of $\Torus$.
\end{lemma}

\begin{remark}\label{rem: dP6 non-minimal S3}
	The dihedral group $\Dih_6=\langle r,s\rangle$ contains two groups isomorphic to $\Sym_3$, but only for one of them the $G$-invariant Picard number is 1, namely for $\langle r^2,s\rangle$, which we denote $\Sym_3^{\rm min}$. The group $\langle r^2,rs\rangle$ will be denoted $\Sym_3^{\rm nmin}$. This is the quotient of $\Dih_6$ by its centre $\Center(\Dih_6)\simeq\Cyc_2$.
\end{remark}

The following lemma shows that we do not lose generality by choosing particular actions of the groups $\Cyc_6$, $\Sym_3$ and $\Dih_6$.

\begin{lemma}\label{lem: dP6 standard actions}
	Let $S$ be a $G$-del Pezzo surface of degree 6, and $G$ be isomorphic to $\Cyc_6$, $\Sym_3$ or~$\Dih_6$. Then $G$ is conjugate in $\Aut(S)$ to the groups $\langle\rho\rangle$, $\langle \rho^2,\sigma\rangle$ and $\langle \rho,\sigma\rangle$, respectively. 
\end{lemma}
\begin{proof}
	This is the content of \cite[Propositions 5.6, 5.7, 5.8]{Pinardin}. We give a short direct proof here for the largest of these groups $G\simeq\Dih_6$. We may assume that it is generated by two elements $\overline{\rho}$ and $\overline{\sigma}$ whose images in $\Aut(\Sigma)$ are as in (\ref{eq: generators of D6}). The map $\overline{\rho}$ is given by
	\begin{equation}
		\overline{\rho}\colon ([x_0:x_1:x_2],[y_0:y_1:y_2])\mapsto ([u y_1:v y_2:y_0],[u^{-1} x_1:v^{-1}x_2:x_0] )
	\end{equation}
	for some $u,v\in\kk^*$. The map 
	\[
	\beta\colon([x_0:x_1:x_2],[y_0:y_1:y_2])\mapsto ([x_0:v^{-1}x_1:uv^{-1}x_2],[y_0:v y_1,vu^{-1}y_2])
	\]
	then conjugates $\overline{\rho}$ to $\rho$. After this conjugation, $\overline{\sigma}$ is given by
	\[
	\overline{\sigma}\colon ([x_0:x_1:x_2],[y_0:y_1:y_2])\mapsto ([ay_0:by_2:y_1],[a^{-1}x_0:b^{-1}x_2:x_1])
	\]
	for some $a,b\in\kk^*$. The condition $\overline{\sigma}^2=\id$ implies $b=1$, while the condition $\overline{\sigma}\circ\rho\circ\overline{\sigma}=\rho^{-1}$ implies $a=1$.  
\end{proof}

\begin{lemma}\label{lem: dP6 fixed points}
	Let $S$ be a $G$-del Pezzo surface of degree 6. Then $G$ fixes a point on $S$ if and only if $G\cap\Torus=\{\id\}$.
\end{lemma}
\begin{proof}
	Assume that $G$ fixes a point on $S$, but $G\cap \Torus\ne\{\id\}$. Since $\Torus$ can be identified with a subgroup of $\PGL_3(\kk)$ which fixes 3 points on $\PP^2$, an element of $\Torus$ fixing a point on $S\setminus\Sigma$ is necessarily trivial. Therefore, a fixed point $p\in S$ of $G$ lies on $\Sigma$. But this implies $\rk\Pic(S)^G>1$.
	
	Conversely, suppose that $G\cap\Torus=\{\id\}$. Then $G$ is isomorphic to $\Cyc_6$, $\Sym_3$ or $\Dih_6$ by Lemma~\ref{lem: dP6 minimal groups}. It remains to apply Lemma~\ref{lem: dP6 standard actions} and notice that groups mentioned there fix the point $([1:1:1],[1:1:1])\in S$.
\end{proof}

\begin{lemma}\label{lem: dP6 fixed point is invariant}
	Let $\varphi\colon S_1\dashrightarrow S_2$ be a $G$-birational map of $G$-del Pezzo surfaces of degree 6. Then $G$ fixes a point on $S_1$ if and only if $G$ fixes a point on $S_2$.
\end{lemma}
\begin{proof}
	It is enough to prove the necessity. Let $\varphi\colon (S_1,G,\iota_1)\dashrightarrow (S_2,G,\iota_2)$ be a $G$-birational map, and assume that $\iota_1(G)$ fixes a point on $S_1$. By Lemma \ref{lem: dP6 fixed points}, the group $\iota_1(G)$ does not intersect the torus $\Aut^\circ(S_1)$ and hence is isomorphic to $\Cyc_6$, $\Sym_3$ or $\Dih_6$ by Lemma \ref{lem: dP6 minimal groups}. But this implies that $\iota_2(G)$ does not intersect the torus $\Aut^\circ(S_2)$: otherwise $\iota_2(G)/(\iota_2(G)\cap\Aut^\circ(S_2))$ is isomorphic to $\Cyc_2,\Cyc_3,\Cyc_2^2,\Sym_3^{\rm nmin}$ (see Remark \ref{rem: dP6 non-minimal S3}), hence $S_2$ is not a $G$-del Pezzo surface. 
\end{proof}

We now discuss (non)linearizability of three particular groups: $\langle\rho\rangle\simeq\Cyc_6$, $\langle \rho^2,\sigma\rangle\simeq\Sym_3$ and $\langle\rho,\sigma\rangle\simeq\Dih_6$. 

\begin{example}\label{ex: dP6 C6 is linearizable}
	The cyclic group $G=\langle\rho\rangle\simeq\Cyc_6$ is linearizable. Indeed, the Sarkisov $G$-link centred at the $G$-fixed point $([1:1:1],[1:1:1])\in S$ leads to the $G$-del Pezzo surface $S'\simeq\PP^1\times\PP^1$. Since $G$ is cyclic, it fixes a point on $S'$ (of course, this can be seen directly as in the next example), hence the stereographic projection from it linearizes the action of $G$.
\end{example}

\begin{example}\label{ex: dP6 S3 is linearizable}
	The symmetric group $G=\langle\rho^2,\sigma\rangle\simeq\Sym_3$ is linearizable as well. Indeed, after the same $G$-link at the fixed point $([1:1:1],[1:1:1])$, we arrive to $S'=\PP^1\times\PP^1$ acted on by~$G$. Recall that $\Aut(S')\simeq(\PGL_2(\kk)\times\PGL_2(\kk))\rtimes\Cyc_2$. Since $\Pic(S')^G=1$, the group $G$ maps non-trivially to $\Cyc_2$ factor of this semi-direct product, i.e. $G=\langle r, s\rangle$, where $r$ is an automorphism of order $3$ acting fibrewisely on $S'$ and $s$ is an involution which does not preserve the rulings of $S'$. It will be shown below in Proposition \ref{prop: dP8 Pic=1 cyclic} that such a group fixes a point on $S'$, and the stereographic projection from this point linearizes the action of $G$.
\end{example}

\begin{example}[{\emph{Iskovskikh's example}}]\label{ex: dP6 Iskovskikh's example}
	Answering the question of V. Popov, V. Iskovskikh showed in \cite{IskovskikhNonConjugate} that the group $G\simeq\Dih_6$ generated by the automorphisms $\rho$ and $\sigma$ is not linearizable. The proof relied on the classical Sarkisov theory. Recently, another proof was given by B. Hassett, A. Kresch and Yu. Tschinkel \cite[Section 7.6]{HassettKreschTschinkel} by using the Burnside group $\Burn_2(G)$ and later by Yu. Tschinkel, K. Yang and Z.~Zhang by using the combinatorial Burnside groups \cite{TschinkelYangZhang}. We refer to the Section \ref{subsec: Burnside} for more information.
\end{example}

We are ready to summarize the results of this section.

\begin{proposition}\label{prop: dP6 criterion}
	Let $S$ be a $G$-del Pezzo surface of degree $6$. Then $G$ is linearizable if and only if $G\simeq\Cyc_6$ or $G\simeq\Sym_3$.
\end{proposition}
\begin{proof}
	The sufficiency follows from Examples \ref{ex: dP6 C6 is linearizable} and \ref{ex: dP6 S3 is linearizable} and Lemma \ref{lem: dP6 standard actions}. To prove the necessity, we use the classification of $G$-links given above. Assume that $G$ is linearizable. Since Bertini and Geiser involutions lead to a $G$-isomorphic surface, there exists a sequence of $G$-links $S\dashrightarrow S_1\dashrightarrow \ldots\dashrightarrow S_n=S'$, where all $S_i$ are $G$-del Pezzo surfaces of degree 6 and $G$ fixes a point on $S'$ --- so that we construct a $G$-link to $\PP^1\times\PP^1$. But then $G$ fixes a point on $S$ by Lemma \ref{lem: dP6 fixed point is invariant}. By Lemma \ref{lem: dP6 fixed points}, $G$ does not intersect the torus $\Aut^\circ(S)$, and hence $G$ maps isomorphically to a subgroup of $\Dih_6$. Now the result follows from Lemma \ref{lem: dP6 standard actions} and Examples \ref{ex: dP6 C6 is linearizable}, \ref{ex: dP6 S3 is linearizable} and \ref{ex: dP6 Iskovskikh's example}.
\end{proof}

\begin{remark}[{\emph{Cayley groups and stable linearizability}}]\label{rem: Cayley groups}
	The (non)linearizability of $\Dih_6$ and its subgroups was addressed in a different context by N. Lemire, V. Popov and Z. Reichstein in their seminal article \cite{LemirePopovReichstein}. They called a connected linear algebraic group $G$ over a field $\KK$ a \emph{Cayley group} if it admits a \emph{Cayley map}, i.e. a $G$-equivariant birational isomorphism between the group variety $G$ and its Lie algebra $\mathrm{Lie}(G)$. Let
	\rowcolors{1}{}{}
	\[
	T=\left \{
	A=\begin{bmatrix}
		a_1 & 0 & 0\\
		0 & a_2 & 0\\
		0 & 0 & a_3
	\end{bmatrix}
	\ \colon\ \det A=1
	\right \},\ \ \
	\mathfrak{t}=\left \{A=
	\begin{bmatrix}
		a_1 & 0 & 0\\
		0 & a_2 & 0\\
		0 & 0 & a_3
	\end{bmatrix}
	\ \colon\ \tr A=0
	\right \}
	\]
	be a maximal torus of $\SL_3(\CC)$ and its Lie algebra. Both are obviously acted on by the group $W=\Sym_3$ permuting $a_1$, $a_2$ and $a_3$. Furthermore, $T$ admits a natural compactification $\{x_0y_0z_0=x_1y_1z_1\}\subset\Proj\CC[x_0,x_1]\times\Proj\CC[y_0,y_1]\times\Proj\CC[z_0,z_1]$, which is a $W$-del Pezzo surface of degree 6. By Example \ref{ex: dP6 S3 is linearizable}, and as was also shown in \cite[9.1]{LemirePopovReichstein}, $T$ and $\mathfrak{t}$ are $W$-birational. By the Corollary of \cite[Lemma 3.5]{LemirePopovReichstein}, this is equivalent to showing that $\SL_3(\kk)$ is a Cayley group, because $W$ is the Weyl group of $\SL_3(\CC)$. 
	
	By the same principle, $\mathbb{G}_2$ is \emph{not} Cayley: its Weyl group is $W=\Dih_6$, the maximal torus and its Lie algebra are $W$-isomorphic to $T$ and $\mathfrak{t}$, respectively. But the latter two are not $W$-birational, as showed in Iskovskikh's Example \ref{ex: dP6 Iskovskikh's example}. Interestingly enough, by \cite[Proposition 9.1]{LemirePopovReichstein} $\mathbb{G}_2\times\mathbb{G}_m^2$ \emph{is} a Cayley group: the varieties $T\times\AA^2$ and $\mathfrak{t}\times\AA^2$ are $\Dih_6$-birational; in other words, $\Dih_6$ is \emph{stably linearizable}. It has recently been shown that one can replace $\AA^2$ with $\AA^1$ here \cite[Proposition 12]{BohningBothmerGrafTschinkel}.
\end{remark}

\begin{remark}\label{rem: dP6 S4 stably linearizable action}
	According to Proposition \ref{prop: dP6 criterion}, any finite group $G$ with $G\cap\Torus\ne\{\id\}$ is not linearizable. In \cite{HassettTschinkelTorsors}, B. Hassett and Yu. Tschinkel give an example of stably linearizable $\Sym_4$-action on the sextic del Pezzo surface (here, one has $G\cap\Torus\simeq\Klein_4$). 
\end{remark}

\section{Finite groups acting on smooth quadric surfaces}\label{sec: quadrics}

Let $S=\PP^1\times\PP^1$. Denote by $\swap$ the involution which permutes two rulings of $S$. The connected component $\Aut(S)^\circ$ of the identity is the subgroup of the automorphisms, which preserve both rulings. This subgroup is isomorphic to $\PGL_2(\kk)\times\PGL_2(\kk)$. Then there is a short exact sequence
\begin{equation}\label{eq: P1xP1 exact sequence}
	\begin{tikzcd}
		1
		\ar{r}
		& 
		\PGL_2(\kk)\times\PGL_2(\kk)
		\ar{r}
		& 
		\Aut(S) 
		\ar{r}{\psi}
		& 
		\langle\sigma\rangle
		\ar{r}
		& 
		1,
	\end{tikzcd}
\end{equation}
which splits as the semidirect product $
\Aut(S)\simeq(\PGL_2(\kk)\times\PGL_2(\kk))\rtimes\langle\swap\rangle,
$
with $\langle\swap\rangle\simeq\Cyc_2$ acting on $\PGL_2(\kk)\times\PGL_2(\kk)$ by permuting two factors.

\begin{proposition}\label{prop: exact sequence p1xp1}
	Let $S=\PP^1\times\PP^1$ and $G\subset\Aut(S)$.
	\begin{enumerate}
		\item If $\rk\Pic(S)^G\simeq\ZZ$, then $G$ fits into the short exact sequence
		\begin{equation}\label{eq: P1xP1 Picard 1 ses}
			\begin{tikzcd}
				1 & {H\times_QH} & G & {\Cyc_2} & 1,
				\arrow[from=1-1, to=1-2]
				\arrow[from=1-2, to=1-3]
				\arrow[from=1-3, to=1-4]
				\arrow[from=1-4, to=1-5]
			\end{tikzcd}
		\end{equation}
		where $H$ is a finite subgroup of $\PGL_2(\kk)$.
		\item If $\rk\Pic(S)^G\simeq\ZZ^2$, then $G\simeq H_1\times_QH_2$, where $H_1$, $H_2$ are finite subgroups of~$\PGL_2(\kk)$. 
	\end{enumerate}
\end{proposition}
\begin{proof}
Let us restrict the exact sequence (\ref{eq: P1xP1 exact sequence}) to $G$. Clearly, $\psi(G)=\{\id\}$ if and only if $G$ preserves both rulings of $S$, which is equivalent to $\Pic(S)^G\simeq\ZZ^2$. If this is the case, then $G$ is a subgroup of $\PGL_2(\kk)\times\PGL_2(\kk)$ and the claim follows from Goursat's lemma~\ref{lem: Goursat} and Proposition \ref{prop: Klein}.

Suppose that $\psi(G)=\langle\swap\rangle$ and let $K=\Ker\psi|_G$. Let $p_1,p_2\colon\PGL_2(\kk)^2\to\PGL_2(\kk)$ be the projections onto the first and the second factors, respectively. By Goursat's lemma, we have $K=H_1\times_Q H_2$, where $H_i=p_i(K)$; in particular, every element of $K$ is of the form $(x,y)\mapsto (h_1x,h_2y)$, where $(x,y)$ are local coordinates on $S$, $h_1\in H_1$, $h_2\in H_2$. The whole group $G$ is the union of two cosets, $K$ and $\tau K$, where $\tau(x,y)=(Ay,Bx)$ for some fixed automorphisms $A,B\in\Aut(\PP^1)$. Since $K$ is normal in $G$, then the conjugation by $\tau$ sends each automorphism $(h_1,h_2)\in K$ onto $(Ah_2A^{-1},Bh_1B^{-1})$. Therefore, $H_1$ is isomorphic to $H_2$ via the map
$
h_1\mapsto Bh_1B^{-1}.
$
Indeed, it is clearly an injective homomorphism. To check surjectivity, we use that $\tau^2=(AB,BA)\in K$, i.e. $AB\in H_1$. So,  for any $h_2\in H_2$ we have that $Ah_2A^{-1}\in H_1$ and hence $(AB)^{-1}A h_2A^{-1}AB=B^{-1}h_2 B\in H_1$ is sent to $h_2$.

\end{proof}

\begin{notation}
	Everywhere below, the matrices $\Rot_n,\A,\B,\C,\D,\E$ and $\F$ are defined as in Section~\ref{subsec: PGL2}, and $\I$ stands for the identity matrix. The elements of $\Aut(S)\setminus\Aut(S)^\circ$, i.e. those which do not preserve the rulings of $S$ and which act as $(x,y)\mapsto (\M y,\N x)$ with $\M,\N\in\PGL_2(\kk)$, will be denoted $(\M,\N,\swap)$. Otherwise, we will write $(\M,\N,\id)$.
\end{notation}

Every fibre product $G=H_1\times_Q H_2$, where $H_1,H_2$ are subgroups of $\PGL_2(\kk)$, obviously acts on $S=\PP^1\times\PP^1$ and forces $\Pic(S)^G\simeq\ZZ^2$. By contrast, not every group $G$ fitting the exact sequence of the form \eqref{eq: P1xP1 Picard 1 ses}
with $H\subset\PGL_2(\kk)$ actually embeds into $\Aut(S)$, as the following example shows.

\begin{example}
	The group $G\simeq\Cyc_2^5$ obviously fits the exact sequence \eqref{eq: P1xP1 Picard 1 ses} with $H\simeq\Klein_4$ and $Q=\{\id\}$. However, it cannot be embedded\footnote{In fact, this group cannot be embedded even into the Cremona group~$\Cr_2(\kk)$, see \cite{BeauvilleElementary}.}  in $\Aut(\PP^1\times\PP^1)$. Indeed, by Lemma~ \ref{lem: finite normalizers} below, the normalizer of $H=\langle \A,\B\rangle$ in $\PGL_2(\kk)$ is $\langle \A,\B\rangle\rtimes\langle\C,\D\rangle\simeq\Klein_4\rtimes\Sym_3\simeq\Sym_4$. Let $g=(\M,\N,\swap)\in G$ be an element such that $\psi(g)=\swap$. Then $G$ is generated by $H\times H$ and $g$. Thus, by multiplying $g$ by an element of $H\times H$, we may assume that $g=(\T_1,\T_2,\swap)$, where $\T_1,\T_2\in\langle\C,\D\rangle$. But $g^2=(\T_1\T_2,\T_2\T_1,\id)\in H\times H$, which forces $\T_1=\T_2\in\{\I,\D,\C\D,\D\C\}$, or $(\T_1,\T_2)\in\{(\C,\C^2),(\C^2,\C)\}$. In all these cases, the group $G$ is isomorphic to $\Klein_4\wr\Cyc_2$. 
\end{example}

So, our next goal is to fill the existing gap in the literature: we characterize completely those finite groups which admit a faithful action on $S=\PP^1\times\PP^1$.

\begin{theorem}\label{thm: groups acting on p1xp1}
	Let $G\subset\Aut(\PP^1\times\PP^1)$ be a finite subgroup. Then $G$ is conjugate to a subgroup of one of the following groups:
	\begingroup
	\renewcommand*{\arraystretch}{1.2}
	\begin{longtable}{|l|l|l|l|}
		\hline
		\rowcolor{light-green}
		GAP & Order & Isomorphism class & Generators \\ \hline
		\textup{No id} & $48n$& $\Dih_n\times\Sym_4$ & $(\Rot_n,\I,\id),(\B,\I,\id),(\I,\A,\id),(\I,\B,\id),(\I,\C,\id),(\I,\D,\id)$ \\ \hline
		\textup{No id} & $120n$& $\Dih_n\times\Alt_5$ & $(\Rot_n,\I,\id),(\B,\I,\id),(\I,\E,\id),(\I,\F,\id)$ \\ \hline
		$[1440,5848]$ & $1440$& $\Sym_4\times\Alt_5$ & $(\A,\I,\id),(\B,\I,\id),(\C,\I,\id),(\D,\I,\id),(\I,\E,\id),(\I,\F,\id)$ \\ \hline
		\textup{No id} & $8n^2$& $\Dih_n\wr\Cyc_2$, $n\geqslant 3$ & $(\Rot_n,\I,\id),(\B,\I,\id),(\I,\I,\swap)$ \\ \hline
		$[1152,157849]$&$1152$&$\Sym_4\wr\Cyc_2$& $(\A,\I,\id),(\B,\I,\id),(\C,\I,\id),(\D,\I,\id),(\I,\I,\swap)$\\ \hline
		\textup{No id} &$7200$&$\Alt_5\wr\Cyc_2$&$(\E,\I,\id),(\F,\I,\id),(\I,\I,\swap)$\\ \hline
	\end{longtable}
	\endgroup
\end{theorem}

The proof of this theorem will be given at the end of the~Section. We will begin with a complete description, in terms of matrix generators, of the subgroups of $\Aut(\PP^1\times\PP^1)$ that satisfy the exact sequence ~\eqref{eq: P1xP1 Picard 1 ses} with $H$ isomorphic to $\Alt_4$, $\Sym_4$, or $\Alt_5$.

\begin{lemma}\label{lem: finite normalizers}
	Let us consider $\Klein_4$, $\Alt_4$, $\Sym_4$ and $\Alt_5$ as subgroups of $\PGL_2(\kk)$. Then their normalizers in $\PGL_2(\kk)$ are the following:
	\begin{enumerate}
		\item $\Norm_{\PGL_2(\kk)}(\Klein_4)\simeq\Sym_4$;
		\item $\Norm_{\PGL_2(\kk)}(\Alt_4)\simeq\Sym_4$;
		\item $\Norm_{\PGL_2(\kk)}(\Sym_4)\simeq\Sym_4$;
		\item $\Norm_{\PGL_2(\kk)}(\Alt_5)\simeq\Alt_5$.
	\end{enumerate}
\end{lemma}
\begin{proof}
	Recall that all finite subgroups of $\PGL_2(\kk)$ are unique up to conjugation. Since the group $\Alt_4$ is normal in $\Sym_4$, then its normalizer contains $\Sym_4$. Similarly, $\Sym_4$ contains a normal copy of $\Klein_4$, hence the normalizer of $\Klein_4$ in $\PGL_2(\kk)$ contains $\Sym_4$. Further, the normalizer of $\Sym_4$ contains $\Sym_4$, while the normalizer of $\Alt_5$ contains $\Alt_5$. Notice that the normalizer of a finite group in $\PGL_2(\kk)$ is an algebraic group, hence we can use their classification provided by Theorem~\ref{thm: algebraic subgrops of PGL2}. In each case, the normalizer is clearly not the whole $\PGL_2(\kk)$, and not the groups of type (3) or (4), as those are metabelian by Remark \ref{rem: structure of algebraic subgroups of PGL2} and hence cannot contain a copy of $\Sym_4$ or $\Alt_5$. The claim follows.
\end{proof}

Next, we show how to conjugate fibred products $H\times_Q H$ to some ``standard forms'' in $\Aut(S)$. Recall that $\Aut(\Alt_5)\simeq\Sym_5$ and one has $\Aut(\Alt_5)/\Inn(\Alt_5)\simeq\Cyc_2$. Any outer automorphism of $\Alt_5$ is a conjugation by an odd permutation in $\Sym_5$. 

\begin{proposition}\label{prop: fibred product standard forms}	
	Let $H$ be one of the groups $\Alt_4$, $\Sym_4$ or $\Alt_5$.
	\begin{enumerate}
		\item\label{standard form A4 S4} Suppose that $H\simeq\Alt_4$ or $H\simeq\Sym_4$. Then every quotient $Q=H/K$ of $H$ defines uniquely the subgroup $K$ and the fibre product $H\times_Q H$, which is conjugate in $\Aut(S)$ to the group
		\begin{equation}\label{eq: fibre product P1xP1}
		H\times_Q H=\{(h_1,h_2)\in H\times H\colon \overline{h}_1=\overline{h}_2\}=\{(h,hk)\colon h\in H, k\in K\}.
		\end{equation}
		In particular, any action of $H$ on $S$ is conjugate to the diagonal one.
		\item\label{standard form A5} Suppose that $H\simeq\Alt_5$ and fix an outer automorphism $\xi\in\Aut(\Alt_5)\setminus\Inn(\Alt_5)$. Then any fibre product $H\times_Q H$ is conjugate in $\Aut(S)$ to one of the following groups:
		\begin{enumerate}
			\item[(a)] $\{(h_1,h_2)\in\Alt_5\times\Alt_5\}$$;$
			\item[(b)] $\{(h,h)\colon h\in\Alt_5\}$, i.e. $\Alt_5$ acts diagonally on $S$$;$
			\item[(c)] $\{(h,\xi(h))\colon h\in\Alt_5\}$.
		\end{enumerate}  
	\end{enumerate}
\end{proposition}
\begin{proof}
	Let $K\subseteq H$ be a normal subgroup. The classification of normal subgroups of $H$ implies that the isomorphism type of the quotient $Q=H/K$ determines $K$ uniquely. Vice versa, the isomorphism type of $K$ determines the quotient $Q$. Therefore, by Goursat's Lemma \ref{lem: Goursat}, the fibred product $H\times_Q H$ is uniquely determined by $H$, $Q$ and an automorphism $\delta\in\Aut(Q)$, and in this case
	\[
	H\times_Q H=\{(h_1,h_2)\in H\times H\colon \delta(\overline{h}_1)=\overline{h}_2\}
	\]
	
	(\ref{standard form A4 S4}) Let us show that one can always obtain $\delta=\id$ by conjugating this group in $\Aut(S)$. We may assume that $Q\ne\{\id\}$. Then direct computations show that one has $H=K\rtimes Q$ for some complement $Q$ to $K$ in $H$; we then identify the quotient $Q$ with a subgroup of $H$ and can write $H=q_1K\sqcup\ldots \sqcup q_nK$, where $Q=\{q_1,\ldots,q_n\}\subset H$. More precisely, one of the following holds:
	\begin{enumerate}
		\item[(i)] $H=\Alt_4$, $K=\{\id\}$, $Q=\Alt_4$ and $\Aut(Q)\simeq\Sym_4$.
		\item[(ii)] $H=\Alt_4$, $K=\Klein_4$, $Q=\langle(123)\rangle\simeq\Cyc_3$ and $\Aut(Q)\simeq\Cyc_2$.
		\item[(iii)] $H=\Sym_4$, $K=\{\id\}$, $Q=\Sym_4$ and $\Aut(Q)=\Inn(Q)\simeq\Sym_4$. 
		\item[(iv)] $H=\Sym_4$, $K=\Klein_4$, $Q=\langle(123),(12)\rangle\simeq\Sym_3$ and $\Aut(\Sym_3)\simeq\Inn(\Sym_3)\simeq\Sym_3$.
		\item[(v)] $H=\Sym_4$, $K=\Alt_4$, $Q=\langle (12)\rangle\simeq\Cyc_2$.
	\end{enumerate}
	In particular we see that any automorphism $\delta\in \Aut(Q)$ is a conjugation $g\mapsto hgh^{-1}$ by an element $h\in \Sym_4$ (note that in the case (ii) the only non-trivial automorphism $(123)\mapsto (123)^2=(132)$ is the conjugation by $(12)$). Since every such $h$ corresponds to an automorphism $\alpha_h\in\Aut(\PP^1)$, the automorphism $(h_1,h_2)\mapsto (h_1,\alpha_h^{-1}h_2\alpha_h)$ conjugates $H\times_Q H$ to the fibre product with $\delta=\id$.
	
	(\ref{standard form A5}) Since $\Alt_5$ is simple, then either $Q=\{\id\}$ or $Q=\Alt_5$. In the first case, we get the group~(a). In the second case one has $H\times_QH=\{(h,\delta(h))\in\Alt_5\times\Alt_5\}$ for $\delta\in\Aut(\Alt_5)$. If $\delta\in\Inn(\Alt_5)$, the same argument as in (\ref{standard form A4 S4}) shows that one can conjugate $H\times_Q H$ to the group~(b). Otherwise $\delta=\gamma\circ\xi$ for some $\gamma\in\Inn(\Alt_5)$. Since $\gamma$ corresponds to the conjugation by an automorphism of $\PP^1$, we can conjugate the whole group to (c).
\end{proof}

\begin{proposition}\label{prop: a4 rank 1}
	Assume that $H\simeq\Alt_4$. Then, up to conjugation in $\Aut(S)$, there are the following cases for $G$, and only them.
	\begingroup
	\renewcommand*{\arraystretch}{1.2}
	\begin{longtable}{|l|l|l|l|}
		\hline
		\rowcolor{light-green}
		GAP & Order & Isomorphism class & Generators \\ \hline
		
		$[24,13]$&$24$&$\Alt_4\times\Cyc_2$&$(\A,\A,\id),(\B,\B,\id),(\C,\C,\id),(\I,\I,\swap)$\\ \hline
		$[24,12]$&$24$&$\Sym_4$&$(\A,\A,\id),(\B,\B,\id),(\C,\C,\id),(\D,\D,\swap)$\\ \hline
		$[96,70]$&$96$&$\Cyc_2^4\rtimes\Cyc_6$&$(\A,\I,\id),(\B,\I,\id),(\C,\C,\id),(\I,\I,\swap)$\\ \hline
		$[96,70]$&$96$&$\Cyc_2^4\rtimes\Cyc_6$&$(\A,\I,\id),(\B,\I,\id),(\C,\C,\id),(\I,\C,\swap)$\\ \hline
		$[96,227]$&$96$&$\Cyc_2^2\rtimes\Sym_4$&$(\A,\I,\id),(\B,\I,\id),(\C,\C,\id),(\D,\D,\swap)$\\ \hline
		$[288,1025]$&$288$&$\Alt_4\wr\Cyc_2$& $(\A,\I,\id),(\B,\I,\id),(\C,\I,\id),(\I,\I,\swap)$ \\ \hline
		$[288,1025]$&$288$&$\Alt_4\wr\Cyc_2$& $(\A,\I,\id),(\B,\I,\id),(\C,\I,\id),(\D,\D,\swap)$\\ \hline
	\end{longtable}
	\endgroup
\end{proposition}
\begin{proof}
	Any subgroup of $\PGL_2(\kk)$ isomorphic to $\Alt_4$ is conjugate to the group $H$ generated by the matrices $A,B,$ and $C$, see Section \ref{subsec: PGL2}. We then apply Proposition \ref{prop: fibred product standard forms}. Note that $K=\Klein_4\triangleleft\Alt_4$ is generated by $A$ and $B$. Then, because of \eqref{eq: fibre product P1xP1}, we have one of the following possibilities for $H\times_Q H$:
	\begin{description}
		\item[$(1)\ K=\{\id\}$]\ $\langle (\A,\A),(\B,\B),(\C,\C)\rangle$;
		\item[$(2)\ K=\Klein_4$] $\langle (\A,\I),(\B,\I),(\I,\A),(\I,\B),(\C,\C)\rangle $
		\item[$(3)\ K=\Alt_4$] $\langle; (\A,\I),(\B,\I),(\C,\I),(\I,\A),(\I,\B),(\I,\C)\rangle $.
	\end{description}
	It remains to identify an element $g\in G$ such that $\psi(g)=\swap$. It is of the form $(\M,\N,\swap)\in\PGL_2(\kk)^2\rtimes\Cyc_2,$ and normalizes $H\times_QH$. Since $(\M,\N,\swap)^{-1}=(\N^{-1},\M^{-1},\swap)$, then for any $\P\in H$ we get that
	\begin{equation}\label{eq: normalization}
		(\N^{-1},\M^{-1},\swap)(\P,\P,\id)(\M,\N,\swap)=(\N^{-1}\P\N,\M^{-1}\P\M,\id)
	\end{equation}
	is an element of $H\times_Q H\subset H\times H$, which implies that $\M$ and $\N$ normalize $H$ in $\PGL_2(\kk)$. By Lemma \ref{lem: finite normalizers}, the normalizer of $H$ in $\PGL_2(\kk)$ is the group isomorphic to $\Sym_4$ which contains~$H$. Consider three cases.
	
	Suppose that $H\times_QH$ is of the form (1) among the three above possibilities. If $\M\in H$, then, up to a multiplication by an element of $H\times_QH$, we may assume that $g=(\I,\N,\swap)$. Therefore, (\ref{eq: normalization}) implies $\N^{-1}\P\N=\P$ for all $\P\in H$. But the centre of $\Alt_4$ is trivial, hence $\N=\I$. If $M\notin H$, then (\ref{eq: normalization}) implies that $\N^{-1}\P\N=\M^{-1}\P\M$ for all $\P\in H$, hence $\M\N^{-1}\in\Centralizer_{\Sym_4}(\Alt_4)=\{\id\}$, and we conclude that $\M=\N$, $g=(\M,\M,\sigma)$. Note that $\M=\T\D$ for $\T\in H$ and $\D$ the matrix introduced in Section \ref{subsec: PGL2} (whose coset generates $\Sym_4/\Alt_4$). Multiplying further $g$ by $(\T^{-1},\T^{-1},\id)$, we achieve $g=(\D,\D\,\swap)$.
	
	Now suppose that $H\times_QH$ is of the form (2). As before, if $\M\in H$, then we may assume $\M=\I$. The condition (\ref{eq: normalization}) gives that $(\N^{-1}\P\N,\P)$ is an element of $H\times_Q H$ for any $\P\in H$, hence $[\N,\P]\in K$ for all $\P\in H$. This implies $\N\in\Alt_4$, hence $\N=\T$, $\N=\T\C$ or $\N=\T\C^2$, where $\T\in\Klein_4$. Identifying $\langle (\I,\A),(\I,\B)\rangle$ with $\Klein_4$, we may further multiply $g$ by an element of this group to get $g=(\I,\I,\swap)$, $g=(\I,\C,\swap)$ or $g=(\I,\C^2,\swap)$, respectively. But we can exclude $g=(\I,\C^2,\swap)$, because is the same as $g=(\I,\C,\swap)$, up to conjugation by $(\D,\D,\id)$, which normalizes $H\times_QH$.
	
	Now, if $\M\notin H$, then (\ref{eq: normalization}) implies that $(\N^{-1}\P\N,\M^{-1}\P\M,\id)$ is an element of $H\times_Q H$ for all $\P\in H$, which is equivalent to $\N^{-1}\P^{-1}\N\M^{-1}\P\M$ being an element of $K$ for all $\P\in H$. But this holds if and only if $[\P,\M\N^{-1}]\in\N K\N^{-1}=K$ for all $\P\in H$. As above, we conclude that $\M\N^{-1}\in\Alt_4$ and hence $g=(\M,\T\M,\swap)$, where $\T\in\Alt_4$. As above, we can multiply $g$ by an element of $\langle (\A,\I),(\B,\I),(\I,\A),(\I,\B)\rangle$ to get $g=(\M,\M,\swap)$, $g=(\M,\C\M,\swap)$ or $g=(\M,\C^2\M,\swap)$. Since $\M=\D\T'$ for some $\T'\in H$, then with a further multiplication by $({\T'}^{-1},{\T'}^{-1},\id)\in H\times_Q H$, we achive $g\in\{(\D,\D,\swap),(\D,\C\D,\swap),(\D,\C^2\D,\swap)\}$. But we can exclude $g=(\D,\C\D,\swap)$ and $g=(\D,\C^2\D,\swap)$, because the squares of those elements are respectively $(\C^2,\C,\id)$ and $(\C,\C^2,\id)$, which are not in $H\times_QH=\langle (\A,\I),(\B,\I),(\I,\A),(\I,\B),(\C,\C)\rangle$.
	
	Finally, assume that $H\times_QH$ is of type (3). Note that $g^2=(\M\N,\N\M,\id)\in H\times H$, therefore $\N=\M^{-1}\T$, where $\T\in\Alt_4$. Multiplying $g$ by $(\I,\T^{-1},\id)$, we may assume $g=(\M,\M^{-1},\swap)$. Finally, since $\M=\T'$ or $\M=\T'\D$ for $\T'\in\Alt_4$, multiplication by an element of $H\times H$ gives $g=(\I,\I,\swap)$ or $g=(\D,\D,\swap)$, respectively.
\end{proof}

\begin{remark}
	Some generators of $G$, obtained during the proofs of Propositions \ref{prop: a4 rank 1}, \ref{prop: s4 rank 1}, may be redundant. This is taken into account in the \emph{Generators} column of the corresponding tables. For example, as soon as $G$ contains $(\I,\I,\swap)$, the group $G$ also contains $(\I,\I,\swap)(\M,\N,\id)(\I,\I,\swap)=(\N,\M,\id)$ for each $(\M,\N,\id)\in G$. Similarly, suppose that we are in case (2) of the above proof, and $G\simeq\Cyc_2^2\rtimes\Sym_4$ is generated by $(\A,\I,\id)$, $(\B,\I,\id)$, $(\I,\A,\id)$, $(\I,\B,\id)$, $(\C,\C,\id)$ and $(\D,\D,\swap)$. The relations in $\Sym_4$ then yield $(\I,\B,\id)=(\I,\D\B\D,\id)=(\D,\D,\swap)^{-1}(\B,\I,\id)(\D,\D,\swap)$, and $(\I,\A,\id)=[(\D,\D,\swap)^{-1}(\A,\I,\id)(\D,\D,\swap)]\cdot (\I,\B,\id)^{-1}$. We leave other cases to the reader. 
\end{remark}

\begin{proposition}\label{prop: s4 rank 1}
	Assume that $H\simeq\Sym_4$. Then, up to conjugation in $\Aut(S)$, there are the following cases for $G$, and only them.
	\begingroup
	\renewcommand*{\arraystretch}{1.2}
	\begin{longtable}{|l|l|l|l|}
		\hline
		\rowcolor{light-green}
		GAP & Order & Isomorphism class & Generators \\ \hline
		
		$[48,48]$&$48$&$\Sym_4\times\Cyc_2$&$(\A,\A,\id),(\B,\B,\id),(\C,\C,\id),(\D,\D,\id),(\I,\I,\swap)$\\ \hline
		$[192,955]$&$192$&$\Cyc_2^4\rtimes\Dih_6$& $(\A,\I,\id),(\B,\I,\id),(\C,\C,\id),(\D,\D,\id),(\I,\I,\swap)$\\ \hline
		$[576,8654]$&$576$&$\Alt_4^2\rtimes\Cyc_2^2$& $(\A,\I,\id),(\B,\I,\id),(\C,\I,\id),(\D,\D,\id),(\I,\I,\swap)$\\ \hline
		$[576,8652]$&$576$&$\Alt_4^2\rtimes\Cyc_4$& $(\A,\I,\id),(\B,\I,\id),(\C,\I,\id),(\D,\D,\id),(\I,\D,\swap)$\\ \hline
		$[1152,157849]$&$1152$&$\Sym_4\wr\Cyc_2$&$(\A,\I,\id),(\B,\I,\id),(\C,\I,\id),(\D,\I,\id),(\I,\I,\swap)$\\ \hline
	\end{longtable}
	\endgroup
\end{proposition}
\begin{proof}
	First, we apply Proposition \ref{prop: fibred product standard forms}. Proper normal subgroups of $\Sym_4$ are $\Alt_4$ and $\Klein_4$, hence the only possibilities for $H\times_QH$ are the following:
	\begin{description}
		\item[$(1)\ K=\{\id\}$]\ $\langle (\A,\A),(\B,\B),(\C,\C),(\D,\D)\rangle $;
		\item[$(2)\ K=\Klein_4$] $\langle (\A,\I),(\B,\I),(\I,\A),(\I,\B),(\C,\C),(\D,\D)\rangle $;
		\item[$(3)\ K=\Alt_4$] $\langle (\A,\I),(\B,\I),(\C,\I),(\I,\A),(\I,\B),(\I,\C),(\D,\D)\rangle $;
		\item[$(4)\ K=\Sym_4$]\ $\langle (\A,\I),(\B,\I),(\C,\I),(\D,\I),(\I,\A),(\I,\B),(\I,\C),(\I,\D)\rangle$.
	\end{description}
	We now look for $g=(\M,\N,\swap)$ such that $G$ is generated by $H\times_QH$ and $h$, i.e. $\psi(g)=\swap$. By Lemma \ref{lem: finite normalizers}, the normalizer of $H$ in $\PGL_2(\kk)$ is $H$ itself. Therefore, $\M,\N\in H$. Up to a multiplication by an element of $H\times_QH$, we may assume that $\M=\I$ and hence $g=(\I,\N,\swap)$. We again argue case by case.
	
	If $H\times_Q H$ is of type (1), then the normalization condition (\ref{eq: normalization}) implies that $\N\in\Center(\Sym_4)=\{\id\}$. If $H\times_Q H$ is of type (2), then $[\N,\P]\in K$ for all $\P\in H$, which implies $\N\in\Klein_4$. Multiplying $g$ by an element of $\langle (\I,\A),(\I,\B)\rangle\simeq\Klein_4$, we get $g=(\I,\I,\swap)$. In the case (3), recall that the quotient of $\Sym_4$ by $\langle \A,\B,\C\rangle\simeq\Alt_4$ is generated by the coset of $\D$. Therefore, multiplying $g$ by an element of $\langle (\I,\A),(\I,\B),(\I,\C)\rangle$, we achieve $g=(\I,\I,\swap)$ or $g=(\I,\D,\swap)$. Finally, in the case (4), we can always get $g=(\I,\I,\swap)$.
\end{proof}

\begin{proposition}\label{prop: a5 rank 1}
	Assume that $H\simeq\Alt_5$. Then, up to conjugation in $\Aut(S)$, there are the following cases for $G$, and only them.
	\begingroup
	\renewcommand*{\arraystretch}{1.2}
	\begin{longtable}{|l|l|l|l|}
		\hline
		\rowcolor{light-green}
		GAP & Order & Isomorphism class & Generators \\ \hline
		
		$[120,35]$&$120$&$\Alt_5\times\Cyc_2$&$(\E,\E,\id),(\F,\F,\id),(\I,\I,\swap)$\\ \hline
		$[120,34]$&$120$&$\Sym_5$&$(\E,\xi(E),\id),(\I,\I,\swap)$\\ \hline
		\textup{No id} &$7200$&$\Alt_5\wr\Cyc_2$&$(\I,\E,\id),(\I,\F,\id),(\E,\I,\id),(\F,\I,\id),(\I,\I,\swap)$\\ \hline
	\end{longtable}
	\endgroup
\end{proposition}
\begin{proof}
	We again let $g=(\M,\N,\swap)\in G$ be an element such that $\psi(g)=\sigma$. Since $g$ normalizes $H\times_Q H$, we get that $\M$ and $\N$ normalize $H$ in $\PGL_2(\kk)$, hence $\M,\N\in H$ by Lemma~\ref{lem: finite normalizers}. We now apply Proposition~\ref{prop: fibred product standard forms}(\ref{standard form A5}). If $H\times_Q H$ is of type (a), then one can multiply $g$ by an element of $H\times_Q H$ to get $g=(\I,\I,\swap)$. In the case (b), the normalization condition (\ref{eq: normalization}) implies $\M\N^{-1}\in\Centralizer_{\Alt_5}(\Alt_5)=\id$, hence $\M=\N$ and we again can replace $g=(\M,\M,\id)$ by $(\I,\I,\swap)$. Finally, assume that we are in case (c), i.e. $H\times_Q H=\{(\P,\xi(\P))\colon \P\in H)\}$. Multiplying $g$ by $(\M^{-1},\xi(\M^{-1}),\id)$, we may assume $\M=\I$. Since for each $\P\in H$, one has
	\[
	(\I,\N,\sigma)^{-1}(\P,\xi(\P),\id)(\I,\N,\sigma)=(\N^{-1}\xi(\P)\N^{-1},\P,\id)\in H\times_Q H,
	\]
	we get $\P=\xi(\N^{-1}\xi(\P)\N)$. Since $\xi$ is of order 2, then we get that $\N\in\Center(\Alt_5)=\{\id\}$.
\end{proof}

\begin{proof}[Proof of Theorem \ref{thm: groups acting on p1xp1}]
	If $\Pic(S)^G\simeq\ZZ^2$, then $G$ is a subgroup of $H_1\times H_2$, where $H_1$ and $H_2$ are finite subgroups of $\PGL_2(\kk)$ acting fibrewisely on $\PP^1\times\PP^1$. Clearly, we can consider the direct products of cyclic and dihedral groups as subgroups of a group of type $\Dih_n\wr\Cyc_2$ from the Theorem, for a suitable $n$. Furthermore, $\Cyc_n\times\Alt_4$ embeds into $\Dih_n\times\Sym_4$, while $\Cyc_n\times\Alt_5$ embeds into $\Dih_n\times\Alt_5$.
	
	Now suppose that $\Pic(S)^G\simeq\ZZ$ and hence $G$ fits the exact sequence \eqref{eq: P1xP1 Picard 1 ses}. If $H$ is isomorphic to $\Alt_4,\Sym_4$, or $\Alt_5$, then the result follows from Propositions \ref{prop: a4 rank 1}, \ref{prop: s4 rank 1} and \ref{prop: a5 rank 1}. Let $g=(\M,\N,\swap)\in G$ be an element such that $\psi(g)=\swap$. Then $G$ is generated by $H\times_Q H$ and $g$. In particular, $G$ is a subgroup of $\widehat{G}=\langle H\times H,g\rangle$. 
		
	Assume that $H$ is cyclic. Then, up to conjugation in $\PGL_2(\kk)$, the group $H$ is generated by $\mathrm{R}_n$. Since $g$ normalizes $H\times_QH$, we deduce that $\M$ and $\N$ are diagonal or anti-diagonal. On the other hand, $g^2\in H\times_Q H$, i.e. the matrices $\M\N$ and $\N\M$ are diagonal, hence either $\M$ and $\N$ are both diagonal, or both anti-diagonal. Let $(u,v)$ be the affine coordinates on $\PP^1\times\PP^1$, and let $\sigma_0$ be the automorphism $(u,v)\mapsto (v,u)$, i.e. $(\I,\I,\swap)$. Then $g$ is either of the form $g\colon (u,v)\mapsto (av,bu)$, or of the form $g\colon (u,v)\mapsto (av^{-1},bu^{-1})$ for some $a,b\in\kk^*$.
	
	In the first case, the automorphism $\alpha\colon (u,v)\mapsto (u,a^{-1}v)$ commutes with $H\times H$ and conjugates $g$ to $g'\colon (u,v)\mapsto (v,cu)$, where $c=ab$. The group $\alpha^{-1}\widehat{G}\alpha=\langle H\times H,g'\rangle$ contains an isomorphic copy of $G$. Since the automorphism ${g'}^2\colon (u,v)\mapsto (cu,cv)$ belongs to $H\times H$, we get $c=\omega_n^k$ for some $k\in\ZZ$.  Multiplying $g'$ by an automorphism $(u,v)\mapsto (u,c^{-1}v)$, which belongs to $H\times H$, we find that $\alpha^{-1}\widehat{G}\alpha$ is generated by $H\times H$ and $\sigma_0$; in particular, it is isomorphic to $\Cyc_n\wr \Cyc_2$.
	
	In the second case, the automorphism $\beta\colon (u,v)\mapsto (u,v^{-1})$ normalizes $H\times H$ and conjugates $g$ to $g'\colon (u,v)\mapsto (av,b^{-1}u)$. Then $\beta^{-1}\widehat{G}\beta=\langle H\times H,g'\rangle$ contains a copy of $G$, and one repeats the argument from the previous case.
	
	Assume that $H\simeq\Dih_n$. Then $H$ is conjugated in $\PGL_2(\kk)$ to the group generated by $\Rot_n$ and $\B$. We may assume $n\geqslant 3$, since for $H\simeq \Klein_4$ the group $G$ is conjugate to a subgroup of~$\Sym_4\wr\Cyc_2$. Let $\T\in H$ be any element, then there is $\P\in H$ so that $(\T,\P,\id)\in H\times_Q H$. One has 
	\begin{equation}
		(\N^{-1},\M^{-1},\swap)(\T,\P,\id)(\M,\N,\swap)=(\N^{-1}\P\N,\M^{-1}\T\M,\id)\in H\times H,
	\end{equation}
	and hence $\M$ normalizes $H$ in $\PGL_2(\kk)$. Now $C=\langle \mathrm{R}_n\rangle$ is a characteristic subgroup of~$H$ for $n\geqslant 3$, hence it is invariant under the conjugation by $\M$, and therefore $\M$ is either diagonal or anti-diagonal. The same obviously holds for $\N$. Now, by multiplying $g$ by an element of $H\times H$, we may assume that $\M$ and $\N$ are either both diagonal, or both anti-diagonal. We finish the proof by applying the same process as in the cyclic case. 
\end{proof}

\begin{remark}
	Our computations exhibit many non-conjugate embeddings of various isomorphic groups into $\Aut(\PP^1\times\PP^1)$, e.g. of some semi-direct products in~Propositions~\ref{prop: a4 rank 1}, \ref{prop: s4 rank 1}.
\end{remark}

\section{Linearization of finite groups acting on Hirzebruch surfaces}\label{sec: Hirzebruch}

In the final section we study linearization of finite groups acting on Hirzebruch surfaces~$\FF_n$ with $n\geqslant 0$. This section is divided into three parts. In the first part, we study Hirzebruch surfaces $\FF_n$ with $n\geqslant 1$. In the second and third parts, we investigate the linearization of groups acting on the quadric surface $\FF_0\simeq\PP^1\times\PP^1$, according to the $G$-invariant Picard rank. 

\subsection{$G$-Hirzebruch surfaces}\label{sec: G-Hirzebruch}

The goal of this section is to prove the following linearization criterion for Hirzebruch surfaces~$\FF_n$ with $n\geqslant 1$.

\begin{theorem}\label{thm: hirzebruch answer}
	Let $n\geqslant 1$ be an integer and $G\subset\Aut(\FF_n)$ be a finite group acting on~ $\FF_n$, such that $\pi\colon \FF_n\to\PP^1$ is a $G$-conic bundle. Denote by $\widehat{G}$ the image of $G$ in the automorphism group of the base $\Aut(\PP^1)\simeq\PGL_2(\kk)$. Then $G$ is linearizable if and only if one of the following holds:
	\begin{enumerate}
		\item $n$ is odd;
		\item $\widehat{G}$ is cyclic;
		\item $\widehat{G}$ is isomorphic to $\Dih_m$ with $m\in\ZZ$ odd.
	\end{enumerate}
\end{theorem}

Before proving this, we need some easy lemmas. First, recall the following classical definition.

\begin{definition} Let $n\geqslant 0$ be a non-negative integer. 
	\begin{enumerate}
		\item An {\it elementary transformation} of the Hirzebruch surface $\FF_n$ is the following birational transformation. Let $\varphi\colon Y\to \FF_n$ be the blow-up of a point $p$ on a fibre $F$, $\widetilde{F}$ be the strict transform of $F$, $\widetilde{\Sigma}_n$ be the strict transform of the $(-n)$-section $\Sigma_n\subset\FF_n$ and $E$ be the exceptional divisor. We have $(\widetilde{F})^2=(\varphi^*F-E)^2=F^2-1=-1$. Then there is a morphism $\psi\colon Y\to Z$ blowing down $\widetilde{F}$. If $p\notin \Sigma_n$, then $\widetilde{\Sigma}_n^2=\Sigma_n^2=-n$ and $\widetilde{\Sigma}_n$ intersects $\widetilde{F}$ transversely in exactly one point. Thus $\psi(\widetilde{\Sigma}_n)^2=-n+1$ and $Z\simeq\FF_{n-1}$. If $p\in\Sigma_n$, then $\widetilde{\Sigma}_n^2=\Sigma_n^2-1=-n-1$, $\widetilde{\Sigma}_n\cap\widetilde{F}=\varnothing$, so $\psi(\widetilde{\Sigma}_n)^2=-n-1$ and $Z\simeq\FF_{n+1}$. Therefore,
		one has the following diagram for an elementary transformation of $\FF_n$.
		
		\begin{equation}\label{eq: elementary transformation}
			\xymatrix@C+1pc{
				& Y\ar[dl]_{\varphi}\ar[dr]^{\psi} &\\
				\FF_n\ar[d]\ar@{-->}[rr] && {Z=\FF_{n+1}\ \text{or}\ \FF_{n-1}}\ar[d]\\
				\PP^1\ar@{=}[rr] && \PP^1
			}
		\end{equation}
		\item Similarly, we define a \emph{$G$-elementary transformation $\FF_{n}\dashrightarrow\FF_{m}$}, where $G$ is a finite group acting on $\FF_n$. In the diagram (\ref{eq: elementary transformation}), the map $\varphi$ is the blow-up of a $G$-orbit of length $\ell$, and $\psi$ is the contraction of the strict transforms of the fibres through the blown-up points. We assume that no two points of the orbit lie in the same fibre; in what follows, we refer to this condition as to \emph{``conic bundle general position''}. If the blown-up orbit belongs to $\Sigma_n$, then $m={n+\ell}$. 
		\item More generally, given a $G$-conic bundle $\pi\colon S\to\PP^1$, we define a \emph{$G$-elementary transformation} (or a link of type $\II$ between $G$-conic bundles) $\chi\colon S\dashrightarrow S'$, where $\pi'\colon S'\to\PP^1$ is another $G$-conic bundle, as the blow-up of a $G$-orbit on $S$, followed by the contraction of the strict transforms of the fibres through the blown-up points. Once again, we assume that no two points of the orbit lie in the same fibre. Note that $\chi$ does not change the number of singular fibres, i.e. one has $K_S^2=K_{S'}^2$.	
	\end{enumerate}
\end{definition}

\begin{lemma}\label{presn}
	Any subgroup $G\subset\Aut(\FF_n)$ preserves a curve on $\FF_n$ which has self-intersection~$n$ and is disjoint from $\Sigma_n$.
\end{lemma}
\begin{proof}
	Recall that $S=\FF_n$ has the unique section $\Sigma_n\subset\FF_n$ of self-intersection $-n$. Let $F$ be a fibre of $S$. The complete linear system $|\Sigma_n+nF|$ induces the birational morphism
	$
	\phi\colon S\to S'\subset\PP^{n+1},
	$
	which is the contraction of $\Sigma_n$ onto the vertex of the cone $S'$. As $\Sigma_n$ is preserved by~ $G$, this contraction is $G$-equivariant. Writing $\PP^{n+1}=\PP(L\oplus V)$ with $L$ being 1-dimensional vector space corresponding to the unique singular (and hence $G$-fixed) point of $S'$, we see that $\PP(V)$ is a $G$-invariant hyperplane. It intersects $S'$ in a $G$-invariant rational normal curve of degree~$n$. Its preimage on $S$ under $\phi$ is the required curve. 
\end{proof}

Now we will analyse $G$-elementary transformations of Hirzebruch surfaces in greater detail, based on the ``arithmetic'' of possible $G$-orbits; a similar idea was used e.g. in \cite[Lemma B.15]{CheltsovTwoLocalInequalities}.

\begin{lemma}\label{lem: elementary transformations and orbits length}
	Let $n\geqslant 1$ and consider a $G$-conic bundle $\pi\colon\FF_n\to\PP^1$. One has the following:
	\begin{enumerate}
		\item\label{elem transformation 1} Let $\ell$ be the length of one of the orbits under the action of $\widehat{G}$ on $\PP^1$. Then the conic bundle $\mathbb F_n$ is $G$-birational to $\FF_{|n+k\ell|}$, for any $k\in\mathbb Z$.
		\item\label{elem transformation 2} Let $\ell_i$, $i=1,\ldots,r$, be the lengths of orbits under the action of $\widehat{G}$ on $\PP^1$, and let $d=\gcd\{\ell_i\}$. Then $\FF_n$ is birational to $\FF_{|n+kd|}$ for any $k\in\ZZ$. 
	\end{enumerate}
\end{lemma}
\begin{proof}
	(\ref{elem transformation 1}) The blow-up of $\FF_n$ at a $G$-orbit of length $\ell$ contained in $\Sigma_n$ and contraction of the proper transforms of the fibres gives a $G$-birational map to $\FF_{n+\ell}$. Moreover, the action of $G$ on the $-(n+\ell)$-curve of $\FF_{n+\ell}$ is the same as the action of $G$ on the $(-n)$-curve of $\FF_n$. It proves the Lemma for $k\geqslant 0$. Let $C$ be a $G$-invariant $n$-curve, given by Lemma \ref{presn}. Notice that since $C$ and $\Sigma_n$ are both sections of the conic bundle, the action of $G$ on $C$ is the same as the action of $G$ on $\Sigma_n$. The blow-up of $\FF_n$ at a $G$-orbit of length $\ell$ contained in $C$ and contraction of the proper transforms of the fibres gives a $G$-birational map to $\FF_{|n-\ell|}$. Moreover, the action of $G$ on the $-|n-\ell|$-curve of $\FF_{|n-\ell|}$ is the same as the action of $G$ on the $(-n)$-curve of $\FF_n$. We proceed by induction. Suppose that constructed a sequence of $G$-elementary transformations $\FF_n\dashrightarrow\FF_{|n-k\ell|}$, where $k\geqslant 1$. If $n-k\ell\geqslant 0$, then an elementary transformation at a $G$-orbit of $\ell$ points, lying on the $G$-invariant $n$-curve as above, gives a map $\FF_{n-k\ell}\dashrightarrow\FF_{|n-(k+1)\ell|}$. If $n-k\ell<0$ then we use an elementary transformation at a $G$-orbit lying on $\Sigma_{|n-k\ell|}$ to get $\FF_{|n-k\ell|}=\FF_{k\ell-n}\dashrightarrow\FF_{k\ell-n+\ell}=\FF_{|n-(k+1)\ell|}$. 
	
	(\ref{elem transformation 2}) By Bézout's identity, we have 
	\begin{equation}\label{eq: Bezout1}
		kd=a_1\ell_1+\ldots+a_s\ell_s-a_{s+1}\ell_{s+1}-\ldots-\ell_ra_r,
	\end{equation}
	where $0\leqslant s\leqslant r$ and $a_i\geqslant 0$ for all $i=1,\ldots,r$. If $k\geqslant 0$, then, by performing $G$-elementary transformations as in \eqref{elem transformation 1}, we get a sequence 
	\begin{equation}\label{eq: Bezout2}
	\FF_{n}\overset{\chi_0}{\dashrightarrow}\FF_{n+a_1\ell_1}\overset{\chi_1}{\dashrightarrow}\FF_{n+a_1\ell_1+a_2\ell_2}\overset{\chi_2}{\dashrightarrow}\ldots\overset{}{\dashrightarrow}\FF_{n+a_1\ell_1+\ldots+a_s\ell_s}\overset{\chi_s}{\dashrightarrow}\FF_{n+a_1\ell_1+\ldots+a_s\ell_s-a_{s+1}\ell_{s+1}}\overset{\chi_{s+1}}{\dashrightarrow}\ldots,
	\end{equation}
	which results in $\FF_{n+kd}$. Suppose now that $k< 0$, and again write $kd$ in the form \eqref{eq: Bezout1}. Perform the sequence of transformations \eqref{eq: Bezout2} up to the $s$-th step. The map $\chi_s$ then leads to $\FF_{|n+a_1\ell_1+\ldots+a_s\ell_s-a_{s+1}\ell_{s+1}|}=\FF_{a_{s+1}\ell_{s+1}-n-a_1\ell_1-\ldots-a_s\ell_s}$. We can proceed by mapping
	\[
	\FF_{a_{s+1}\ell_{s+1}-n-a_1\ell_1-\ldots-a_s\ell_s}\dashrightarrow \FF_{a_{s+2}\ell_{s+2}+a_{s+1}\ell_{s+1}-n-a_1\ell_1-\ldots-a_s\ell_s}=\FF_{|n+a_1\ell_1+\ldots+a_s\ell_s-a_{s+1}\ell_{s+1}-a_{s+2}\ell_{s+2}|},
	\]
	and so on, until we get $\FF_{|n+kd|}$.
\end{proof}

Now, we deduce several corollaries.

\begin{corollary}\label{cor: Fn cyclic or dihedral}
	If $\widehat{G}$ is cyclic or isomorphic to $\Dih_m$ with $m$ odd, then $\FF_n$ is $G$-birational to $\FF_k$, for any non-negative integer $k$. In particular, $G$ is linearizable.
\end{corollary}
\begin{proof}
	If $\widehat{G}$ is cyclic, then it fixes a point on $\PP^1$. If $\widehat{G}$ is isomorphic to $\Dih_m$ with $m$ odd, then there is an orbit of length 2 and of odd length by Proposition \ref{prop: Klein orbits}. In both cases, by Lemma \ref{lem: elementary transformations and orbits length} (\ref{elem transformation 2}), the surface $\FF_n$ is $G$-birational to $\FF_{|n+s|}$ for any $s\in\ZZ$, hence the first claim. In particular, by taking $s=1-n$, we arrive at $\FF_1$ where we can $G$-equivariantly contract the unique $(-1)$-curve to get $\PP^2$.
\end{proof}

\begin{corollary}\label{cor: Fn odd}
	The $G$-conic bundle $\FF_n$ is $G$-birational to any $\FF_{|n+2k|}$, for any $k\in\ZZ$. So, if $n$ is even, then $\FF_n$ is $G$-birational to the $G$-conic bundle $\FF_0\simeq\PP^1\times\PP^1$. If $n$ is odd, then $\FF_n$ is $G$-birational to $\FF_1$ and hence $G$ is linearizable in this case.
\end{corollary}
\begin{proof}
	If $\widehat{G}$ is cyclic, we conclude using Corollary \ref{cor: Fn cyclic or dihedral}. Otherwise, Proposition \ref{prop: Klein orbits} implies that there is always a pair (or a triple) of orbits of $\widehat{G}$ having $2$ as the greatest common divisor of their lengths. By Lemma \ref{lem: elementary transformations and orbits length} (\ref{elem transformation 2}), we conclude that $\FF_n$ is $G$-birational to $\FF_{|n+2k|}$, for any $k\in\ZZ$.
\end{proof}

\subsection{Quadrics with invariant Picard group of rank $1$}

We will classify finite linearizable subgroups of $\Aut(S)$ according to the isomorphism class of $H$ in the exact sequence~ (\ref{eq: P1xP1 Picard 1 ses}).

\begin{proposition}\label{prop: dP8 Pic=1 cyclic}
	If $H$ is cyclic, then $G$ is linearizable.
\end{proposition}
\begin{proof}
	We claim that $G$ has a fixed point, so the stereographic projection from it linearizes the action of $G$. Indeed, the group $H\times_Q H$ is contained in $H\times H$, hence it has exactly 4 fixed points on $S$. We can choose the coordinates on $S$ so that these points are
	\[
	p_1=([1:0],[1:0]),\ \ p_2=([1:0],[0:1]),\ \ p_3=([0:1],[1:0]),\ \ p_4=([0:1],[0:1]).
	\]	
	Let $g\in G$ be an element which is mapped to the generator of $\Cyc_2$ in the exact sequence (\ref{eq: P1xP1 Picard 1 ses}), i.e. which does not preserve the rulings of $S$. Then $G$ is generated by $g$ and $H\times_Q H$. Since $g$ normalizes $H\times_Q H$, it permutes the points $p_1,p_2,p_3,p_4$. Since $g$ swaps the rulings of $S$, it fixes $p_1$ and $p_3$, or it fixes $p_2$ and $p_4$, or it permutes $p_1$, $p_2$, $p_3$ and $p_4$ cyclically. In the first two cases, we are done. The third case is impossible, since $g^2$ would swap $p_i$ with $p_{i+2}$, while it belongs to $H\times_Q H$, i.e. fixes all four points.
\end{proof}

\begin{remark}\label{rem: linearization via stereo}
	In the language of Sarkisov program, we linearize the action of $G$ as follows. Let $\eta\colon T\to S=\PP^1\times\PP^1$ be the blow-up of a $G$-fixed point $p\in S$. The surface $T$ is a del Pezzo surface of degree 7 with three exceptional curves $F_1, F_2$, and $E$, where $E$ is the $\eta$-exceptional divisor, $F_1$ and $F_2$ are the preimages of fibres through $p$. One has $E\cdot F_1=E\cdot F_2=1$, and $F_1\cdot F_2=0$. Then there is a $G$-contraction $\eta'\colon T\to\PP^2$ of $F_1$ and $F_2$ onto a pair of points.
\end{remark}

We are going to show that $G$ is not linearizable in any of the other cases. 

\begin{observation*}\label{linksp1p1} 	
	We use the Notations of Section \ref{subsec: Sarkisov theory}. Every Sarkisov $G$-link starting from $S$ is either of type $\I$, where $\eta$ blows up a $G$-orbit of length 2, or is of type $\II$, and one of the following holds:
	\begin{enumerate}
		\item $S\simeq S'$, $d=d'=7$, $\chi$ is a birational Bertini involution;
		\item $S\simeq S'$, $d=d'=6$, $\chi$ is a birational Geiser involution;
		\item $S'$ is a del Pezzo surface of degree 5, $d=5$, $d'=2$;
		\item $S'\simeq\PP^1\times\PP^1$, $d=d'=4$;
		\item $S'$ is a del Pezzo surface of degree 6, $d=3$, $d'=1$;
		\item $S'\simeq\PP^2$, $d=1$, $d'=2$.
	\end{enumerate}	
	Birational Geiser and Bertini involutions lead to a $G$-isomorphic surface. By \cite[Proposition 4.3]{YasinskyGBirationalRigidity}, links centred at orbits of length 4 also result in a surface which is $G$-isomorphic to $S$. Therefore, any $G$-link starting from $S$ and leading to a non-isomorphic surface must be centred at an orbit of length $d\in\{1,2,3,5\}$.
\end{observation*}

\begin{proposition}\label{prop: dP8 Pic=1 A4 S4 A5 not linearizable}
	Let $S=\PP^1\times\PP^1$ and $G\subset\Aut(S)$ be a finite group such that $\Pic(S)^G\simeq\ZZ$. Assume that in the setting of the exact sequence \eqref{eq: P1xP1 Picard 1 ses} the group $H$ is isomorphic to $\Alt_4,\Sym_4$ or $\Alt_5$. Then $S$ is $G$-birationally rigid. In particular, $G$ is not linearizable.
\end{proposition}
\begin{proof}
	Note that an orbit of $G$ is a disjoint  union of orbits of $H\times_Q H$. Therefore, by Proposition \ref{prop: Klein orbits} and Lemma~\ref{lem: orbits of a direct product}, any orbit of $G$ on $S$ has length $2m\geqslant 4$. Hence $S$ can admit only $G$-birational Geiser involutions and links at points of degree 4. As was observed before, they give a $G$-isomorphic surface.
\end{proof}

\begin{proposition}\label{prop: dP8 Pic=1 Dihedral not linearizable}
	Let $S=\PP^1\times\PP^1$ and $G\subset\Aut(S)$ be a finite group such that $\Pic(S)^G\simeq\ZZ$. Assume that in the setting of the exact sequence \eqref{eq: P1xP1 Picard 1 ses} the group $H$ is dihedral $\Dih_n$. Then $G$ is not linearizable.
\end{proposition}
\begin{proof}
	
	We may assume that there is a $G$-orbit $\Sigma\subset S$ in general position, i.e. its blow-up gives a del Pezzo surface. Then Lemma \ref{lem: orbits of subgroups} implies that $\Sigma$ is a disjoint union of at most 2 orbits of $H\times_Q H$. Suppose that $\Sigma=\Sigma_1\sqcup\Sigma_2$, where $|\Sigma_1|=|\Sigma_2|$. By Proposition \ref{prop: Klein orbits} and Lemma~\ref{lem: orbits of a direct product}, these cardinalities are divisible by 2 or $n$. Then $|\Sigma|$ is an even number $\geqslant 4$, and therefore $S$ is $G$-birationally rigid. So, we may assume that $\Sigma$ is an orbit of $H\times_Q H$. By the same observations, we only need to investigate orbits of size 2, 3 and 5.
	
	Suppose $|\Sigma|=5$, i.e. $H\simeq\Dih_5$. The link $S\dashrightarrow S'$ centred at $\Sigma$ gives a $G$-del Pezzo surface $S'$ of degree 5. Since $Q$ can be isomorphic to $\{\id\}$, $\Cyc_2$ or $\Dih_5$, we find that either $G$ is an extension of $\Cyc_2$ by $\Dih_5$, or $G$ contains a copy of $\Cyc_5^2$ by Lemma~ \ref{lem: Goursat exact sequence}. In the latter case, $G$ cannot be embedded into $\Aut(S')\simeq\Sym_5$, and neither $\Dih_5\times\Cyc_2$ can. Hence we may assume that $G\simeq\Frobenius_5$. But in this case, $G$ cannot be linearized by Proposition~\ref{dP5: D5 is linearizable}.	
	
	If $|\Sigma|=3$, then $H\simeq\Sym_3$, and the blow-up of $\Sigma$ gives a del Pezzo surface $T$ of degree 5. Since $Q$ can be isomorphic to $\{\id\}$, $\Cyc_2$ or $\Sym_3$, we find that either $G\simeq\Sym_3\times\Cyc_2$, or $G$ contains a copy of $\Cyc_3^2$ by Lemma~ \ref{lem: Goursat exact sequence}. In the latter case, $G$ does not embed into $\Aut(T)\simeq\Sym_5$, hence we may assume that we are in the former case. Then our Sarkisov link ends up on a $G$-del Pezzo surface of degree~6. Since $G\simeq\Sym_3\times\Cyc_2$, it cannot be linearized by Proposition~\ref{prop: dP6 criterion}.
	
	Finally, case $|\Sigma|=2$ corresponds to a link of type $\I$ leading to the del Pezzo surface $T$ of degree~$6$ with a $G$-conic bundle structure. All Sarkisov $G$-links starting from a $G$-conic bundle with~2 singular fibres are either $G$-elementary transformations $T\dashrightarrow T'$, or links $T'\to S'$ of type $\III$ leading back to a $G$-del Pezzo surface $S'$ of degree~8, so the composition of such links looks like in Diagram \ref{eq: sequence of links}. In view of Propositions \ref{prop: dP8 Pic=1 cyclic}, \ref{prop: dP8 Pic=1 A4 S4 A5 not linearizable} and the previous cases, it is enough to show that $G$ does not fix a point on $S'$, and hence cannot be further linearized. Let us denote cyclically the six $(-1)$-curves forming a hexagon on~$T$ by $E_1,\dots,E_6$. We may assume that $E_1$ and $E_4$ are the exceptional divisors over $\Sigma$, so in particular the action of~$G$ must swap $E_1$ and $E_4$. They both form sections of the $G$-conic bundle $\pi\colon T\rightarrow\PP^1$. On the other hand, $E_2+E_3$ and $E_5+E_6$ are the two singular fibres of $\pi$, whose irreducible components are swapped by $G$. By Theorem \ref{theorem: minimal conic bundles}, the $G$-conic bundle $\pi'\colon T'\to\PP^1$ is again a del Pezzo surface of degree 6. Denote by $D_1$ and $D_2$ the singular fibres of $\pi'$. Besides their irreducible components, the surface $T'$ has two more $(-1)$-curves $L_1$ and $L_2$, which are disjoint sections of $\pi'$, swapped by $G$. Note that the points $E_2\cap E_3$ and $E_5\cap E_6$ are unique $G$-fixed points on $T$. Since a sequence of $G$-elementary transformations $T\dashrightarrow T'$ is an isomorphism away from $G$-orbit of curves of length $>1$, the singular points of $D_1$ and $D_2$ are the unique $G$-fixed points on $T'$. Thus, the contraction $T'\to S'$ of $L_1$ and $L_2$ leaves no $G$-fixed points on $S'$.
\end{proof}

\subsection{Quadrics with invariant Picard group of rank $2$}

Suppose that $\Pic(S)^G\simeq\ZZ^2$. We endow $S=\PP^1\times\PP^1$ with coordinates $([x_0:x_1],[y_0:y_1])$ and fix two projections $\pi_1$ and $\pi_2$ to $[x_0:x_1]$ and $[y_0:y_1]$, respectively. We often use the affine coordinates $x=x_1/x_0$ and $y=y_1/y_0$ on $S$, so that the points $([1:0],[1:0])$ and $([0:1],[0:1])$ correspond to $(0,0)$ and $(\infty,\infty)$. We follow Notations \ref{notation: matrices} for the generators of finite subgroups of $\PGL_2(\kk)$.

By Proposition \ref{prop: exact sequence p1xp1}, the group $G$ is the fibred product $H_1\times_Q H_2$, where $H_1$ and $H_2$ are finite subgroups of $\PGL_2(\kk)$, acting on $S$ fibrewisely (projections of $G$ induced by $\pi_1$ and $\pi_2$, respectively). The crucial observation, which follows from \cite[Propositions 7.12, 7.13]{DolgachevIskovskikh}, is that $G$ is linearizable if and only if there is a sequence of Sarkisov $G$-links
\[
S=\FF_0\overset{\chi_0}{\dashrightarrow} S_1\overset{\chi_1}{\dashrightarrow} S_2\overset{\chi_2}{\dashrightarrow}S_3\overset{\chi_3}{\dashrightarrow}\ldots \overset{\chi_{N-1}}{\dashrightarrow} S_N=\FF_1\overset{\chi_N}{\longrightarrow}\PP^2,
\]
where $S_i\simeq\FF_{n_i}$ for all $i$, the links $\chi_i$ for $i=0,\ldots,N-1$ are either $G$-elementary transformations of Hirzebruch surfaces, or links of type $\IV$ (in which case $S_i\simeq S_{i+1}\simeq\PP^1\times\PP^1$), and the last link $\chi_N$ is necessarily of type $\III$, namely the equivariant blow-down of the unique $(-1)$-curve on $\FF_1$ to the $G$-fixed point on $\PP^2$.

The main result of this subsection is the following.

\begin{theorem}\label{thm: P1xP1 Pic=2 linearization}
	The group $G=H_1\times_Q H_2$ is linearizable if and only if one of the following holds:
	\begin{enumerate}
		\item\label{c1} Both $H_1$ and $H_2$ are cyclic;
		\item\label{c2} $H_1$ is cyclic, and $H_2\simeq\Dih_{n}$, where $n\geqslant 1$ is odd, or vice versa;
		\item\label{c3} $H_1\simeq\Dih_{n},H_2\simeq\Dih_{m}$, where $n,m\geqslant 3$ are odd, and $G$ is isomorphic to a dihedral group.
	\end{enumerate}
\end{theorem}

We split the proof into several auxiliary statements.

\begin{lemma}\label{lemma: cycliccyclic}
	The group $G=\Cyc_n\times_Q\Cyc_m$ is linearizable for any $n,m\in\ZZ_{>0}$.
\end{lemma}
\begin{proof}
	We may assume that $G$ is a subgroup of $\langle \mathrm{R}_n\rangle\times\langle \mathrm{R}_m\rangle$, and hence fixes the point $([1:0],[1:0])$; the stereographic projection from this point linearizes $G$.
\end{proof}

\begin{lemma}\label{lemma: odd dihedral cyclic}
	Let $n$ and $m$ be integers, with $n$ odd. The group $G=\Dih_n\times_Q\Cyc_m$ is linearizable.
\end{lemma}
\begin{proof}
	The group $\Cyc_m$ fixes a point on $\PP^1$, and hence $G$ acts on the fibre $F$ over this point. Therefore, the image of $H_1\times_Q H_2$ in $\Aut(F)\simeq\PGL_2(\kk)$ is isomorphic to $\Dih_n$ with $n$ odd. By Proposition \ref{prop: Klein orbits}, there is a $G$-orbit of size $n$ on $F$. A $G$-elementary transformation centred at such orbit leads to the Hirzebruch surface $\FF_n$ and we are done by Corollary~\ref{cor: Fn odd}.
\end{proof}

\begin{lemma}\label{lemma: when P1xP1 not linearizable}
	If $H_1$ or $H_2$ is isomorphic to $\Alt_4,\Sym_4,\Alt_5$, or $\Dih_{2n}$, then $G$ is not linearizable.
\end{lemma}
\begin{proof}
	We start with the following observation. Consider a $G$-elementary transformation
	\[
	\xymatrix{
		&T\ar@{->}[dl]_{\eta}\ar@{->}[dr]^{\eta'}&\\
		\FF_{2n}\ar@{-->}[rr]^{\chi}\ar[dr]&& \FF_m\ar[dl]\\
		& {\PP^1} &}
	\]
	centred at a $G$-orbit $p_1,\ldots,p_{2k}\in\FF_{2n}$ of length $2k\geqslant 2$. Let us show that $m$ is even. Indeed, assume that $m$ is odd. For $i=1,\ldots, 2k$, let $F_i\subset\FF_{2n}$ be the fibres through the blown-up points~$p_i$, let $E_i=\eta^{-1}(p_i)\subset T$ be the exceptional divisors over these points, $\widetilde{F}_i\subset T$ be strict transforms of $F_i$, and $q_i=\eta'(F_i)$ be the resulting $G$-orbit on $\FF_m$. Suppose that the curve $\Sigma_m\subset\FF_m$ contains $t\geqslant 0$ points among $q_i$. Its strict transform $\widetilde{\Sigma}_m\subset T$ under $\eta'$ then satisfies $\widetilde{\Sigma}_m^2=-m-t$, it intersects exactly $t$ curves $\widetilde{F}_i$ on $T$, and $2k-t$ curves $E_i$ on $T$. The $\eta$-images of these $2k-t$ curves are the blown-up points lying on $C=\eta(\widetilde{\Sigma}_m)$. Thus, $C^2-(2k-t)=\widetilde{\Sigma}_m^2=-m-t$, hence $C^2=2k-2t-m$ is odd, which is impossible\footnote{Recall that $\Pic(\FF_N)$ is generated by the class of a fibre $F$ and the $(-N)$-section $\Sigma_N$ (or a transversal fibre $\Sigma_0$, such that $\Sigma_0\cdot F=1$, when $N=0$). One has $F^2=0$, $F\cdot\Sigma_N=1$, $\Sigma_N^2=-N$. So, for a curve $C\sim aF+b\Sigma_N$ one has $C^2=2ab-Nb$.} on $\FF_{2n}$.
		
	Now let $H_1$ and $H_2$ be as in the condition of the Lemma. It is sufficient to show that there is no sequence of $G$-elementary transformations (possibly alternating with links of type IV on $\FF_0$) from $S=\FF_0$ to $\FF_1$. By Proposition \ref{prop: Klein orbits} and Lemma~\ref{lem: orbits of a direct product}, the group $G=H_1\times_Q H_2$ has only orbits of even length on $S$. Thus, the first $G$-elementary transformation brings us to some $\FF_{2n}$. Since the lengths of the orbits is preserved under $G$-fibrewise transformations, a sequence of $G$-elementary transformations can only lead to Hirzebruch surfaces $\FF_N$ with $N$ even.
\end{proof}

It remains to study the linearizability of $G$ when $H_1\simeq\Dih_n$ and $H_2\simeq\Dih_m$, with $n$ and $m$ odd. By using Notations \ref{notation: matrices}, we can suppose that $\Dih_n=\langle\Rot_n,\B\rangle$ and $\Dih_m=\langle\Rot_m,\B\rangle$ as subgroups of $\PGL_2(\kk)$. Recall that possible fibre products $\Dih_n\times_Q\Dih_m$ were essentially described in Lemma \ref{lem: subgroups of Dn x Dm}.

\begin{lemma}\label{lemma: fibred product of dihedrals}
	Let $n,m$ be odd positive integers. If $G=\Dih_n\times_Q\Dih_m$ is not isomorphic to a dihedral group, then $G$ is not linearizable.
\end{lemma}
\begin{proof}
If $G$ is linearizable, then there is a sequence of Sarkisov $G$-links $S_1=\PP^1\times\PP^1\dashrightarrow S_2\dashrightarrow\ldots\dashrightarrow S_k=\FF_1\to\PP^2$, where the last map is the $G$-equivariant blow-down of the unique $(-1)$-curve on $\FF_1$ --- a link of type $\III$. The image of this curve is a $G$-fixed point $p\in\PP^2$, and hence $G$ admits a faithful representation in $\GL(T_p\PP^2)\simeq\GL_2(\kk)$. We conclude by Lemma~ \ref{lem: subgroups of Dn x Dm}.
\end{proof}

\begin{proposition}[Linearization via the Euclidean algorithm]\label{prop: Euclid}
	Let $n$ and $m$ be odd positive integers. Suppose that the group $G=\Dih_n\times_Q\Dih_m$ is dihedral. Then $G$ is linearizable.
\end{proposition}
\begin{proof}
	For convenience, we set $\Dih_1=\Cyc_2$, so that $G=\Dih_n\times_{\Dih_d}\Dih_m$, where $d\geqslant 1$ divides both $n$ and $m$. By Lemma \ref{lem: Goursat exact sequence}, there is a short exact sequence of groups
	\begin{equation}
		\begin{tikzcd}
			1
			\ar{r}
			& 
			\langle(\Rot_n^d,\id) \rangle\times\langle(\id,\Rot_m^d)\rangle
			\ar{r}
			& 
			G
			\ar{r}{\varrho}
			& 
			\Dih_d
			\ar{r}
			& 
			1.
		\end{tikzcd}
	\end{equation}
	Since $G$ is dihedral, then $\ker\varrho$ must be cyclic, so $\gcd(n/d,m/d)=1$ and $\ker\varrho$ is generated by the element $(\Rot_n^d,\Rot_m^d)$. By Remark \ref{Generators of a fibre product}, the group $G$ is generated by $(\Rot_n^d,\Rot_m^d)$ and $(\B,\Rot_m^v\B)$ when $d=1$, and by $(\Rot_n^d,\Rot_m^d)$, $(\Rot_n,\Rot_m^u)$ and $(\B,\Rot_m^v\B)$ when\footnote{Note that the second case includes the sub-case $n=m=d$. As follows from Goursat's lemma, we have $G\simeq\{(g,\varphi(g))\colon g\in\Dih_n\}$, where $\varphi\in\Aut(\Dih_n)$. It is well known that in the standard presentation of $\Dih_n$, its automorphism group is generated by the maps $r\mapsto r^j$, $s\mapsto r^ts$, which agrees with the generators we chose.} $d>1$; here, $u,v$ are some positive integers. Let $k=v/2$ if $v$ is even and $k=(v+m)/2$ if $v$ is odd. By conjugating $G$ in $\Aut(S)$ by the automorphism $(\id,\Rot_m^k)$, we may assume $v=0$. 
	
	Let $M=nm$. Then
	\[
	(\Rot_n^d,\Rot_m^d)=(\Rot_M^{dm},\Rot_M^{dn}),\ \ (\Rot_n,\Rot_m^u)=(\Rot_M^m,\Rot_M^{nu}).
	\]
	Regardless of whether $d=1$ or $d>1$, the elements of $G$ are all of the form $(\Rot_M^a,\Rot_M^b)$ or $(\B\Rot_M^a,\B\Rot_M^b)$ for some positive integers $a,b$. Since the latter ones are involutions, we can assume that the characteristic cyclic subgroup of $G$ is generated by an element of the form $(\Rot_M^a,\Rot_M^b)$, i.e. by the map $(x,y)\mapsto (\omega_M^ax,\omega_M^by)$.
	
	Consider the following birational self-map of $S$ and its inverse:
	\[
	\varphi\colon (x,y)\dashmapsto (x,x^{-1}y),\ \ \varphi^{-1}\colon (x,y)\dashmapsto (x,xy).
	\]
	Given a biregular automorphism $g\colon (x,y)\mapsto (\alpha x,\beta y)$ of $S$ for some $\alpha,\beta\in\kk^*$, we have the following commutative diagram
	\[
	\xymatrix{
		S\ar@{->}[rr]^{g}\ar@{-->}[d]_{\varphi} && S\ar@{-->}[d]^{\varphi}\\
		S\ar@{->}[rr]^{\overline{g}} && S,		
	}
	\]
	where $\overline{g}\colon (x,y)\mapsto (\alpha x,\alpha^{-1}\beta y)$ is $\varphi\circ g\circ\varphi^{-1}$. The conjugation of the automorphism $(\B,\B)$ by $\varphi$ gives the same action. To sum up, $\varphi$ is a birational equivalence between $(\PP^1\times\PP^1,G)$ and $(\PP^1\times\PP^1,\overline{G})$, where $\overline{G}$ is generated by  
	\[
	(x,y)\mapsto (\omega_M^ax,\omega_M^{b-a}y),\ \
	(x,y)\mapsto (x^{-1},y^{-1}).
	\]
	The biregular automorphism $\sigma\colon (x,y)\mapsto (y,x)$ conjugates the map $(x,y)\mapsto (\alpha x,\beta y)$ to $(x,y)\mapsto (\beta x,\alpha y)$ and does not change $(\B,\B)$. Therefore, up to conjugation by $\sigma$, we may assume $0\leqslant a\leqslant b$. We now run the Euclidean algorithm for $a$ and $b$: first, by iterating the conjugation by $\varphi$, we can replace $(\Rot_M^a,\Rot_M^b)$ by $(\Rot_M^a,\Rot_M^r)$, where $r$ is the remainder of division of $b$ by $a$. By using $\sigma$, we replace $(\Rot_M^a,\Rot_M^r)$ by $(\Rot_M^r,\Rot_M^a)$, perform Euclidean division of $a$ by $r$, and so on. The algorithm results in the generator $(\Rot_M^\ell,\id)$, where $\ell=\gcd(a,b)$. 
	
	Hence, we birationally conjugated $G$ to the dihedral group generated by $(\Rot_M^\ell,\id)$ and $(\B,\B)$. The fibre over $y=1$ is now faithfully acted on by $G$. Making an elementary transformation at the orbit of size $N=|G|/2$ in this fibre, we arrive to $\FF_N$ and hence can further linearize the action of $G$.
\end{proof}

\begin{proof}[Proof of Theorem \ref{thm: P1xP1 Pic=2 linearization}]
	Follows from Lemmas \ref{lemma: cycliccyclic}--\ref{lemma: fibred product of dihedrals} and Proposition \ref{prop: Euclid}.
\end{proof}

\section*{Appendix: Magma code}

The following Magma code determines the GAP ID of finite subgroups of $\Aut(\PP^1\times\PP^1)\simeq\PO(4)$ defined by their matrix generators. The code is also available on the GitHub page of the first author \cite{PinardinGitHub}.

\begin{lstlisting}[language=Magma]
K:=AlgebraicClosure(Rationals());
K4:=CartesianPower(K,4);
R<x>:=PolynomialRing(K);
i:=Roots(x^2+1)[1,1];
w5:=Roots(x^5-1)[2,1];

A:=Matrix(K,2,[1,0,0,-1]);
B:=Matrix(K,2,[0,1,1,0]);
C:=Matrix(K,2,[i,-i,1,1]);
D:=Matrix(K,2,[1,-i,i,-1]);
E:=Matrix(K,2,[w5,0,0,1]);
F:=Matrix(K,2,[1,1-w5-w5^-1,1,-1]);
I:=Matrix(K,2,[1,0,0,1]);

// The following function takes two ordered sets of 4 points of P3 
// in general position, and returns the unique matrix in PGL(4) 
// which takes the first set to the second one.

AutP3:=function(points,images)
	points:=[<t:t in points[1]>,<t:t in points[2]>,<t:t in points[3]>,<t:t in points[4]>,<t:t in points[5]>];
	images:=[<t:t in images[1]>,<t:t in images[2]>,<t:t in images[3]>,<t:t in images[4]>,<t:t in images[5]>];

	A1<a,b,c,d>:=AffineSpace(K,4);
	X1:=Scheme(A1,[a*points[2][1]+b*points[3][1]+c*points[4][1]+d*points[5][1]-points[1][1],a*points[2][2]+b*points[3][2]+c*points[4][2]+d*points[5][2]-points[1][2],a*points[2][3]+b*points[3][3]+c*points[4][3]+d*points[5][3]-points[1][3],a*points[2][4]+b*points[3][4]+c*points[4][4]+d*points[5][4]-points[1][4]]);

	sols1:=RationalPoints(X1);
	eltsM1:=[];

	for i in [1..4] do
		eltsM1:=eltsM1 cat ElementToSequence(sols1[1][i]*Vector(K,4,[points[i+1][1],points[i+1][2],points[i+1][3],points[i+1][4]]));
	end for;
    M1:=Transpose(Matrix(K,4,eltsM1));
	
	X2:=Scheme(A1,[a*images[2][1]+b*images[3][1]+c*images[4][1]+d*images[5][1]-images[1][1],a*images[2][2]+b*images[3][2]+c*images[4][2]+d*images[5][2]-images[1][2],a*images[2][3]+b*images[3][3]+c*images[4][3]+d*images[5][3]-images[1][3],a*images[2][4]+b*images[3][4]+c*images[4][4]+d*images[5][4]-images[1][4]]); 

	sols2:=RationalPoints(X2);
	eltsM2:=[];
	
	for i in [1..4] do
		eltsM2:=eltsM2 cat ElementToSequence(sols2[1][i]*Vector(K,4,[images[i+1][1],images[i+1][2],images[i+1][3],images[i+1][4]]));
	end for;

	M2:=Transpose(Matrix(K,4,eltsM2));
	M:=M2*M1^-1;
	return M;
end function;

// The following function takes a triple (M1,M2,s) in Aut(P1xP1), 
// and returns the matrix of PGL4 whose restriction on the quadric xt=yz 
// gives the corresponding automorphism after the Segre embedding.

AutQuadSurf:=function(M1,M2,s)
	M1:=Matrix(K,2,ElementToSequence(M1));
	M2:=Matrix(K,2,ElementToSequence(M2));
	if s eq 0 then
		g:=map<K4->K4|x:-><(ElementToSequence(M1*Matrix(K,2,1,[x[1],x[2]])) cat ElementToSequence(M2*Matrix(K,2,1,[x[3],x[4]])))[1],(ElementToSequence(M1*Matrix(K,2,1,[x[1],x[2]])) cat ElementToSequence(M2*Matrix(K,2,1,[x[3],x[4]])))[2],(ElementToSequence(M1*Matrix(K,2,1,[x[1],x[2]])) cat ElementToSequence(M2*Matrix(K,2,1,[x[3],x[4]])))[3],(ElementToSequence(M1*Matrix(K,2,1,[x[1],x[2]])) cat ElementToSequence(M2*Matrix(K,2,1,[x[3],x[4]])))[4]>>;
	else
		g:=map<K4->K4|x:-><(ElementToSequence(M1*Matrix(K,2,1,[x[3],x[4]])) cat ElementToSequence(M2*Matrix(K,2,1,[x[1],x[2]])))[1],(ElementToSequence(M1*Matrix(K,2,1,[x[3],x[4]])) cat ElementToSequence(M2*Matrix(K,2,1,[x[1],x[2]])))[2],(ElementToSequence(M1*Matrix(K,2,1,[x[3],x[4]])) cat ElementToSequence(M2*Matrix(K,2,1,[x[1],x[2]])))[3],(ElementToSequence(M1*Matrix(K,2,1,[x[3],x[4]])) cat ElementToSequence(M2*Matrix(K,2,1,[x[1],x[2]])))[4]>>;
	end if;

	Phi:=map<K4->K4|x:-><x[1]*x[3],x[1]*x[4],x[2]*x[3],x[2]*x[4]>>;
	PhiInv:=map<K4->K4|x:-><x[1]+x[2],x[3]+x[4],x[1]+x[3],x[2]+x[4]>>;
	g1:=map<K4->K4|x:->Phi(g(PhiInv(x)))>;
	
	points:=[<2,8,1,4>,<1,1,0,0>,<1,0,1,0>,<0,1,0,1>,<0,0,2,1>];
	images:=[g1(points[1]),g1(points[2]),g1(points[3]),g1(points[4]),g1(points[5])];
	return AutP3(points,images);	
end function;

// The following function takes a list of invertible matrices
// and their size, and returns the subgroup of PGL they generate.

SGPGL:=function(matrices,dimension)
    gens:=[];
    for M in matrices do
        gens:=Append(gens,M/Roots(x^dimension-R!Determinant(M))[1,1]);
    end for;
    G:=sub<GL(dimension,K)|[GL(dimension,K)|M: M in gens]>;
    D:=[M: M in Center(G) | IsScalar(M)];
    GP:=quo<G|D>;
    return(GP);
end function;

SGAutP1P1:=function(triples)
	matrices:=[AutQuadSurf(t[1],t[2],t[3]):t in triples];
	return SGPGL(matrices,4);
end function;

// In what follows, IdentifyGroup is the built-in function 
// which returns the GAP ID of a finite group, if it exists.
// Checking the generators given in Notation 3.6 for A4, S4, and A5:

print "GAP ID of the subgroup of PGL(2) generated by A,B and C:", IdentifyGroup(SGPGL([A,B,C],2));
print "GAP ID of the subgroup of PGL(2) generated by A,B,C and D:", IdentifyGroup(SGPGL([A,B,C,D],2));
print "GAP ID of the subgroup of PGL(2) generated by E and F:", IdentifyGroup(SGPGL([E,F],2)),"\n";

// Checking the GAP IDs in Proposition 5.10:

print "GAP IDs of the subgroups of Aut(P^1xP^1) given in Proposition 5.10:\n";

IdentifyGroup(SGAutP1P1([<A,A,0>,<B,B,0>,<C,C,0>,<I,I,1>]));
IdentifyGroup(SGAutP1P1([<A,A,0>,<B,B,0>,<C,C,0>,<D,D,1>]));
IdentifyGroup(SGAutP1P1([<A,I,0>,<B,I,0>,<C,C,0>,<I,I,1>]));
IdentifyGroup(SGAutP1P1([<A,I,0>,<B,I,0>,<C,C,0>,<I,C,1>]));
IdentifyGroup(SGAutP1P1([<A,I,0>,<B,I,0>,<C,C,0>,<D,D,1>]));
IdentifyGroup(SGAutP1P1([<A,I,0>,<B,I,0>,<C,I,0>,<I,I,1>]));
IdentifyGroup(SGAutP1P1([<A,I,0>,<B,I,0>,<C,I,0>,<D,D,1>]));

// Checking the GAP IDs in Proposition 5.13:

print "\n GAP IDs of the subgroups of Aut(P^1xP^1) given in Proposition 5.13:\n";

IdentifyGroup(SGAutP1P1([<A,A,0>,<B,B,0>,<C,C,0>,<D,D,0>,<I,I,1>]));
IdentifyGroup(SGAutP1P1([<A,I,0>,<B,I,0>,<C,C,0>,<D,D,0>,<I,I,1>]));
IdentifyGroup(SGAutP1P1([<A,I,0>,<B,I,0>,<C,I,0>,<D,D,0>,<I,I,1>]));
IdentifyGroup(SGAutP1P1([<A,I,0>,<B,I,0>,<C,I,0>,<D,D,0>,<I,D,1>]));
IdentifyGroup(SGAutP1P1([<A,I,0>,<B,I,0>,<C,I,0>,<D,D,0>,<I,I,1>]));

// Printing a possible description of the last group in Propositon 5.13.
G:=SGAutP1P1([<A,I,0>,<B,I,0>,<C,I,0>,<D,I,0>,<I,I,1>]);
print "Last group. Order: ",Order(G)," description: ",GroupName(G);

\end{lstlisting}
\bigskip 
\bigskip

\def\bibindent{2.5em}

\bibliographystyle{myalpha}
\bibliography{biblio}

\end{document}